\documentclass{amsart}
\usepackage{amsmath}
\usepackage{mathrsfs}
\usepackage{amssymb}
\usepackage{mathtools}
\usepackage{amscd}
\usepackage[all]{xy}
\usepackage{color}
\usepackage{fancyhdr}
\usepackage{hyperref}

\theoremstyle{plain}
\newtheorem{theorem}{Theorem}[section]
\newtheorem{proposition}[theorem]{Proposition}
\newtheorem{lemma}[theorem]{Lemma}
\newtheorem{corollary}[theorem]{Corollary}
\newtheorem*{thm-a}{Theorem A}
\newtheorem*{thm-b}{Theorem B}
\newtheorem*{thm-c}{Theorem C}
\newtheorem*{thm-d}{Theorem D}

\newtheorem*{pro}{Proposition}

\theoremstyle{definition}
\newtheorem{definition}[theorem]{Definition}
\newtheorem{example}[theorem]{Example}

\newtheorem*{problem}{Problem}

\theoremstyle{remark}
\newtheorem{remark}[theorem]{Remark}

\def\A{\mathscr{A}}
\def\B{\mathscr{B}}
\def\C{\mathscr{C}}
\def\T{\mathscr{T}}
\def\F{\mathscr{F}}
\def\GV{{\rm GV}}
\def\Ext{{\rm Ext}}
\def\Hom{{\rm Hom}}
\def\Tor{{\rm Tor}}
\def\pd{{\rm pd}}

\def\iwsd{{\rm \hbox{$iw$-}sd}}
\def\wsd{{\rm \hbox{$w$-}sd}}

\def\m{{\rm \hbox{$R$-}Mod}}

\def\ker{{\rm Ker}~}

\def\cok{{\rm Coker}~}
\def\p{\mathfrak{p}}
\newcommand\bbar[1]{\mkern1.5mu\overline{\mkern-1.5mu{#1}\mkern-1.5mu}\mkern1.5mu}
\def\tor{{\rm tor_{\rm GV}}}
\def\wmax{w\mbox{-}{\rm Max}}

\begin{document}

\title[Cotorsion modules relative to $\tau_w$]{Cotorsion modules relative to a hereditary torsion theory}

\author[Qiao]{Lei Qiao}
\address{(Qiao) School of Mathematical Science, Sichuan Normal University, Chengdu 610066, P. R. China}
\email{lqiao@sicnu.edu.cn}


\thanks{$2020$ Mathematics Subject Classification: 13D07; 13D30; 13C13; 13G05}

\date{\today}

\begin{abstract}
	This paper investigates various classes of cotorsion modules relative to the hereditary torsion theory $\tau_w$ induced by the so-called $w$-operation. To achieve this, we first introduce the $\Tor$ and cotorsion pairs relative to $\tau_w$, which serve as our main tool. Additionally, we also introduce and study the strongly flat modules relative to $\tau_w$. As applications, several known rings (including perfect rings, almost perfect domains, Pr\"{u}fer $v$-multiplication domains, Dedekind domains, Krull domains, and DW-rings) are characterized in terms of these classes of modules.
\end{abstract}

\maketitle

\section*{Introduction and preliminaries}

Recall that an abelian group is called \textit{cotorsion} if every extension of it by a torsion-free group splits. The concept of cotorsion abelian groups was introduced independently by Harrison \cite{H59}, Nunke \cite{Nunke59} and Fuchs \cite{Fuchs60} in 1959/60. Subsequently Matlis \cite{Matlis64} and Warfield \cite{Warfield70} generalized this concept to modules over arbitrary commutative domains in two different ways. If $R$ is a commutative domain with the field of quotients $Q$, then an $R$-module $C$ is said to be \textit{Matlis-cotorsion} (or weakly cotorsion) if $\Ext^1_R(Q,C)=0$, and it is called \textit{Warfield-cotorsion} (or strongly cotorsion) if $\Ext^1_R(M,C)=0$ for every torsion-free $R$-module $M$. Later, in \cite{E84}, Enochs defined cotorsion modules over arbitrary associative rings. If $R$ is an associative ring with unit, then a module $C$ is called \textit{cotorsion} (or Enochs-cotorsion) if $\Ext^1_R(F,C)=0$ for any flat $R$-module $F$. Evidently, over commutative domains, 
\[\text{Warfield-cotorsion}\Rightarrow\text{Enochs-cotorsion}\Rightarrow\text{Matlis-cotorsion},\]
and examples show that in general the reverse implications fail. However, it is well known that a commutative domain $R$ is a Dedekind domain if and only if the three notions of Matlis-cotorsion, Enochs-cotorsion and Warfield-cotorsion modules coincide, and that the notions of Enochs-cotorsion and Warfield-cotorsion modules coincide if and only if $R$ is a Pr\"{u}fer domain (see, for example, \cite[Theorems 4.4.9 and 4.4.10]{GT06}). In \cite{BS02}, Bazzoni and 
Salce characterized the commutative domains $R$ for which the class of Matlis-cotorsion modules coincides with the class of Enochs-cotorsion modules. It was proved that they are exactly the \textit{almost perfect domains}, i.e., they have the property that $R/I$ is a perfect ring (in the sense of Bass \cite{Bass60}) for every nonzero ideal $I$ of $R$. 

It is well known that Dedekind domains play a crucial role in classical algebraic number theory, and their systematic study has furnished substantial impetus for the rapid advancement of commutative ring theory and ideal theory alike. Also, in the introduction to his book, \textit{Multiplicative Ideal Theory} \cite{Gilmer68}, Gilmer states: ``\textit{A third concept which plays a central role in the development of the classical ideal theory is that of a Pr\"{u}fer domain}". Although there are many ways to characterize Dedekind domains and Pr\"{u}fer domains, the cotorsion viewpoint yields elegant homological statements, and these statements precisely establish a connection between the classical ideals-based characterizations and the module-theoretic vanishing conditions. 

Recently, Geroldinger et al. \cite{GKL25} have posed 25 long-term open problems in multiplicative ideal theory and factorization theory; the 14th is as follows:

\begin{problem}(\cite[p. 17]{GKL25})
	Develop explicit cotorsion and Tor-cotorsion characterizations of rings and integral domains that arise in multiplicative ideal theory. In particular, investigate how homological invariants such as
	\[\Ext^1_R(-,-),~\Tor^R_1(-,-),~\text{global dimensions, etc.}\]
	can be used to classify such rings. Examples of interest include: Pr\"{u}fer domains, Krull domains, Pr\"{u}fer $v$-mulitiplication domains, etc.
\end{problem}

The main purpose of this paper is to study various classes of cotorsion modules relative to the hereditary torsion theory $\tau_w$ induced by the so-called $w$-operation, and to use these classes of modules to investigate the Problem.

The notion of cotorsion modules relative to torsion theories was first studied by Henderson and Orzech \cite{HO73}; they replaced the classical notion of ``torsion'' by a torsion theory and generalized some of the results of Matlis \cite{Matlis64} to certain torsion theories over commutative rings. If $R$ is an associative ring with unit and if $\sigma$ is a hereditary torsion theory on the category of unitary $R$-modules, then a left $R$-module $M$ is said to be \textit{$\sigma$-cotorsion} if $\Ext^1_R(Q_{\sigma}(R),M)=0$, where $Q_{\sigma}(R)$ is the ring of quotients of $R$ relative to $\sigma$. Fay \cite{Fay89}, in extending results of de la Rosa and Fuchs \cite{RF86} concerning dimension of torsion divisible modules over a domain, reinvestigated the approach of \cite{HO73} towards cotorsion, obtaining the extension of the duality theorem in Matlis \cite{Matlis64} to commutative rings by replacing the usual notion of torsion in a domain $R$ by a perfect torsion theory $T$ with $TR=0$. Fuchs \cite{Fuchs77} also defined cotorsion in terms of some torsion theory and obtained a structure theorem for cotorsion modules over Notherian hereditary rings analogous to those for abelian groups. Let $R$ be an associative ring with unit, and let $(\T,\F)$ be a hereditary torsion theory of unitary left $R$-modules. Then a left $R$-module $E$ is called \textit{$\T$-cotorsion} if $\Ext^1_R(G,E)=0$ for all $\T$-torsion-free modules $G$. As we see in \cite[p. 182, Definition]{E84}, the definition of Enochs-cotorsion modules agrees with that of Fuchs \cite{Fuchs77}. Almahdi et al. \cite{ATA18} have studied $w$-cotorsion modules (i.e., the elements of the right orthogonal complement of the class of all $w$-flat modules) over an arbitrary commutative ring $R$, where $w$ is the so-called $w$-operation and it can induce the hereditary torsion theory $\tau_w$ over $R$. They prove that the class of $w$-cotorsion modules is strictly contained in the class of Enochs-cotorsion modules and give a new homological characterization of Pr\"{u}fer $v$-multiplication domains in terms of $w$-cotorsion modules. Recently, for a multiplicative subset $S$ of a commutative ring $R$, two different versions of $S$-cotorsion modules were respectively studied by Assaad and Zhang \cite{AZ23} and Bennis and Bouziri \cite{BB25}. In general these two classes of $S$-cotorsion modules are both proper subclasses of the class of Enochs-cotorsion modules over commutative rings.

Next, we shall review some terminology related to the hereditary torsion theory $\tau_w$; see \cite{WK24} for details. The reader should consult the books of Stenstr\"{o}m \cite{S75} and Golan \cite{G86} for background in hereditary torsion theory.

Let $R$ be a commutative ring. Recall that an ideal $J$ of $R$ is called a
\textit{Glaz-Vasconcelos ideal} (a \textit{$\GV$-ideal} for short)
if $J$ is finitely generated and the natural homomorphism
$$\varphi:R\rightarrow J^*:=\Hom_R(J,R)$$ is an isomorphism.
It is clear that a finitely generated ideal $J$ of $R$ is a
$\GV$-ideal if and only if $$\mbox{$\Hom_R(R/J,R)=0$ and
	$\Ext^1_R(R/J,R)=0$.}$$ Notice that the set $\GV(R)$ of all $\GV$-ideals of
$R$ is a multiplicative system of ideals of $R$ (see \cite[Proposition 6.1.9]{WK24}). Let $M$ be an
$R$-module. Define
$$\tor(M):=\{x\in M~|~\mbox{$Jx=0$ for some $J\in\GV(R)$}\}.$$ Thus
$\tor(M)$ is a submodule of $M$. Now $M$ is said to be
\textit{$\GV$-torsion} (resp., \textit{$\GV$-torsion-free}) if
$\tor(M)=M$ (resp., $\tor(M)=0$). A $\GV$-torsion-free module $M$ is
called a \textit{$w$-module} if $\Ext^1_R(R/J,M)=0$ for all
$J\in\GV(R)$. Then it follows from \cite[Theorems 6.7.35 and 6.1.18]{WK24} that flat modules and reflexive modules are both
$w$-modules. Let $M$ be a $\GV$-torsion-free module. Then
$$M_w:=\{x\in E(M)~|~\mbox{$Jx\subseteq M$ for some $J\in\GV(R)$}\}$$
is a $w$-submodule of $E(M)$ containing $M$ and is called the
$w$-\textit{envelope} of $M$, where $E(M)$ denotes the injective
envelope of $M$. It is obvious that $M_w/M$ is the $\GV$-torsion submodule $\tor\left(E(M)/M\right)$ of $E(M)/M$, and that $M$ is a $w$-module if and only if $M_w=M$. Let $\wmax(R)$ denote the set of $w$-ideals of a commutative ring $R$ maximal among
proper integral $w$-ideals of $R$ and we call $\mathfrak{m}\in\wmax(R)$ a
\textit{maximal $w$-ideal} of $R$. Then every proper $w$-ideal is
contained in a maximal $w$-ideal and every maximal $w$-ideal is a
prime ideal. Note that an $R$-module $M$ is $\GV$-torsion if and only if $M_{\mathfrak{m}}=0$ for all $\mathfrak{m}\in\wmax(R)$ (see \cite[Theorem 6.2.13]{WK24}). Let $f:M\rightarrow N$ be a homomorphism of $R$-modules. Then we say that $f$ is a $w$-\textit{monomorphism} (resp., $w$-\textit{epimorphism}, $w$-\textit{isomorphism}) if
$f_\mathfrak{m}:M_\mathfrak{m}\rightarrow N_\mathfrak{m}$ is a monomorphism (resp., an epimorphism, an isomorphism) over $R_\mathfrak{m}$ for any maximal $w$-ideal $\mathfrak{m}$ of $R$. 

In the commutative domain case, $w$-modules were called \textit{semi-divisorial modules} by Glaz and Vasconcelos in \cite{GV77} and (in the ideal case) $F_\infty$-{\it ideals} by Hedstrom and Houston in \cite{HH80}, which have been proved to be useful in the study of ideal theory and module theory. It is worthwhile to point out that ``$w$'' may be a natural bridge between homological algebra and multiplicative ideal theory. On one hand, in the ideal theory, ``$w$'' is better known as the so-called $w$-\textit{operation}, which was introduced by Wang and McCasland \cite{WM97} in order to define the notion of strong Mori domains. Since then, a lot of work has been done on the $w$-operation. As we see in \cite[p. 451]{Zafrullah00}: ``\textit{Now what is so nice about the $w$-operation is that it is smoother than the $t$-operation in that $\cdots$. Smoothness of this sort can only mean one thing, that you can bring in a lot more homological algebra than with other star operations, if that is what you want.}'' For a detailed study of star operations the reader may consult Gilmer's book \cite{Gilmer92}. On the other hand, ``$w$" also can be used to induce the hereditary torsion theory $\tau_w$ (see \cite[p. 1618]{KLZ19}) such that the notion of $w$-modules coincides with that of $\tau_w$-closed (i.e., $\tau_w$-torsion-free and $\tau_w$-injective) modules, where the torsion theory $\tau_w$ whose torsion modules are the $\GV$-torsion modules and the torsion-free modules are the $\GV$-torsion-free modules. 

However, there appears to be no proof in the literature that $\tau_w$ is a hereditary torsion theory. In the following, we will provide one.

\begin{pro}
	Let $\tau_w$ be the pair $(\T_{\GV},\F_{\GV})$, where $\T_{\GV}$ (resp. $\F_{\GV}$) is the class of all $\GV$-torsion (resp. $\GV$-torsion-free) $R$-modules. Then $\tau_w$ is a hereditary torsion theory, and the corresponding Gabriel filter $\mathfrak{F}_{\tau_w}$ consists of the ideals containing a $\GV$-ideal.
\end{pro}
\begin{proof}
	First, note that for every $R$-submodule $L$ of an $R$-module $M$, we have $\tor(L)=L\cap\tor(M)$.
	
	For each $M\in\T_{\GV}\cap\F_{\GV}$, $M=\tor(M)=0$. So $\T_{\GV}\cap\F_{\GV}=\{0\}$.
	
	Let $M\in\T_{\GV}$ having a nonzero homomorphic image $N$ and let $f:M\rightarrow N$ be the canonical surjection. Then for each $y\in N$, $y=f(x)$ for some $x\in M$. Since $M$ is $\GV$-torsion, there is $J\in\GV(R)$ such that $Jx=0$, and so $Jy=f(Jx)=0$. Therefore, $N\in\T_{\GV}$. Thus, $\T_{\GV}$ is closed under taking homomorphic images.  
	
	Let $M\in\F_{\GV}$ and $L$ an $R$-submodule of $M$. Then $\tor(L)=L\cap\tor(M)=L\cap 0=0$, i.e., $L\in\F_{\GV}$. Hence, $\F_{\GV}$ is closed under taking submodules.
	
	For every $R$-module $M$, write $T=\tor(M)$ and $F=M/T$. Clearly, $T\in\T_{\GV}$. Let $J\cdot\bar{x}=0$, where $J\in\GV(R)$, $\bar{x}=x+T\in F$, and $x\in M$. Then $Jx\subseteq T$. Since $J$ is finitely generated and $\GV(R)$ is a multiplicative system of ideals of $R$, there exist $J'\in\GV(R)$ such that $J'Jx=0$. But $J'J\in\GV(R)$, so $x\in T$, i.e., $\bar{x}=0$. Consequently, $F\in\F_{\GV}$. Thus, it follows that $0\rightarrow T\rightarrow M\rightarrow F\rightarrow 0$ is an exact sequence with $T\in\T_{\GV}$ and $F\in\F_{\GV}$.
	
	Let $M\in\T_{\GV}$ and $L$ an $R$-submodule of $M$. Then $\tor(L)=L\cap\tor(M)=L\cap M=L$, i.e., $L\in\T_{\GV}$. Hence, $\T_{\GV}$ is also closed under taking submodules.
	
	Now, combining the above, we see that $\tau_w$ is a hereditary torsion theory. 
	
	Finally, if $I\in\mathfrak{F}_{\tau_w}$, then $R/I\in\T_{\GV}$, and so $J\cdot\bar{1}=0$ for some $J\in\GV(R)$, that is, $J\subseteq I$. Conversely, if $I$ is an ideal of $R$ containing a $\GV$-ideal $J$ of $R$, then $R/J\in\T_{\GV}$ and $R/I$ is a homological image of $R/J$. Hence, $R/I\in\T_{\GV}$. Thus, $\mathfrak{F}_{\tau_w}$ consists of the ideals containing a $\GV$-ideal.
\end{proof}

 In fact, one can further conclude that the hereditary torsion theory $\tau_w$ is well-centered (in the sense of Cahen \cite{Cahen73}). 

Now, we recall the concept of $w$-flat modules (i.e., flat modules relative to $\tau_w$), which appeared first in \cite{Wang97} when the ring is a commutative domain and was extended to arbitrary commutative rings in \cite{KW14}. Recall that an $R$-module $M$ is said to be \textit{$w$-flat} if the induced map $1\otimes f:M\bigotimes_R A\longrightarrow M\bigotimes_R B$ is a $w$-monomorphism for any $w$-monomorphism $f:A\longrightarrow B$. It is shown in \cite[Theorem 3.3]{KW14} that an $R$-module $M$ is $w$-flat if and only if $\Tor^R_1(M,N)$ is a $\GV$-torsion module for any $R$-module $N$, and if and only if $M_{\mathfrak{m}}$ is flat over $R_{\mathfrak{m}}$ for all $\mathfrak{m}\in\wmax(R)$. Certainly, all flat modules are $w$-flat. Moreover, all $\GV$-torsion modules are also $w$-flat, since a module is $\GV$-torsion if and only if its localizations at the maximal $w$-ideals vanish. The class of $w$-flat modules has been used to characterize von Neumann regular rings and Pr\"{u}fer $v$-multiplication domains (see \cite[Theorem 4.4]{WK14} and \cite[Theorem 3.5]{WQ15}).

Recall also from \cite{WK24} that a commutative ring $R$ is called a \textit{DW-ring} if every ideal of $R$ is a $w$-ideal. Equivalently, $R$ is a DW-ring if and only if every maximal ideal of $R$ is a $w$-ideal, if and only if $\GV(R)=\{R\}$. DW-domains were investigated first by Dobbs et al. \cite{DHLZ89,DHLRZ92}; they were mentioned in their papers as \textit{$t$-linkative domains}. In \cite{Mimouni05}, Mimouni was the first to name $t$-linkative domains as DW-domains. In the ideal theory, the terminology of DW-domains reflects the feature that the $d$-operation and the $w$-operation on such a ring are identical. From the homological algebra point of view, DW-rings are exactly the rings of the small finitistic dimensions at most one (see \cite[Corolllary 3.7]{ZW23}). Examples of DW-rings include zero-dimensional rings, rings of weak global dimension at most one, treed domains (in particular, one-dimensional domains), divisorial domains, etc.; see for example \cite[Theorem 6.2.8]{WK24}, \cite[Proposition 2.2]{TAB19}, and \cite[Corollary 2.6]{DHLZ89}. However, a $d$-dimensional Noetherian regular local ring ($d\geq 2$) is never a DW-ring. Indeed, if $(R,\mathfrak{m})$ is such a ring, then it is easy to see that $\mathfrak{m}$ is a $\GV$-ideal of $R$, and hence $R$ is not a DW-ring. 

Throughout this paper, all rings are commutative with an identity element and all modules are unitary; in particular, $R$ denotes such a ring. $\m$ stands for the category of all $R$-modules. $Q$ will denote the classical ring of quotients of $R$, and $K=Q/R$. Any undefined notions or notation are standard, as in \cite{GT06,S75,Rotman79,WK24,Gilmer92}.

Let us outline the structure of this paper. In Section 1, we first discuss the $\Tor$ and cotorsion pairs relative to $\tau_w$ (referred to as $w$-$\Tor$-pairs and $w$-cotorsion pairs, respectively; see Definitions \ref{the definition of w-Tor-pairs} and \ref{the definition of w-cotorsion-pairs}), which serve as our main tool. Proposition \ref{w-cotorsion pair} establishes that a $w$-$\Tor$-pair can simultaneously induce both a classical cotorsion pair and a $w$-cotorsion pair. Section 2 investigates two classes of Enochs-cotorsion modules relative to $\tau_w$: $w$-Enochs-cotorsion modules (Definition \ref{the definition of w-E-cotorsion}) and $\GV$-Enochs-cotorsion modules (Definition \ref{the definition of GV-E-cotorsion}). These classes are employed to characterize distinguished classes of rings, including perfect rings, von Neumann regular rings, and DW-rings. For example, Theorem \ref{perfect ring} states that a ring $R$ is perfect if and only if all $R$-modules are $\GV$-Enochs-cotorsion. In Section 3, we develop the concept of $w$-torsion-free modules (Definition \ref{the definition of w-torsion-free}) to investigate two types of Warfield-cotorsion modules relative to $\tau_w$: $w$-Warfield-cotorsion modules and $\GV$-Warfield-cotorsion modules (Definitions \ref{the definition of w-W-cotorsion} and \ref{the definition of GV-W-cotorsion}, respectively). This section provides several new characterizations of Pr\"{u}fer $v$-multiplication domains and DW-domains, as detailed in Theorem \ref{strongly GV-W-cotorsion=strongly GV-E-cotorsion} and Proposition \ref{w-W-cotrosion=W-cotorsion}, respectively. The final section focuses on $w$-strongly flat modules (Definition \ref{the definition of w-strongly flat}), i.e., strongly flat modules relative to $\tau_w$. We analyze these modules over domains, providing structural characterizations (Theorem \ref{characterization of  strongly w-flat}), localization properties (Proposition \ref{localization of strongly $w$-flat}), and relations to $\GV$-torsion modules (Proposition \ref{GV-torsion strongly w-flat}), etc. Applications include characterizations of Krull domains, Matlis domains, almost perfect domains, and Dedekind domains. For instance, Corollary \ref{a chracterization of Krull domians in terms of strongly $w$-flat modules} proves that a domain $R$ is a Krull domain if and only if every ideal of $R$ is $w$-strongly flat.

We summarize some of our main results in the following:

\begin{thm-a}{\rm (Theorem \ref{GV-E-torsion=w-E-torsion}, Propositions \ref{GV-W-torsion=w-W-torsion} and \ref{strongly w-flat=strongly flat})}
	The following statements are equivalent for a ring $R$.
	\begin{enumerate}
		\item $R$ is a DW-ring.
		\item The classes of all $\GV$-Enochs and $w$-Enochs cotorsion $R$-modules coincide.
		\item The classes of all $\GV$-Warfield and $w$-Warfield cotorsion $R$-modules coincide.
		\item The classes of all $w$-strongly and strongly flat $R$-modules coincide.
	\end{enumerate}
\end{thm-a}

\begin{thm-b}{\rm (Theorem \ref{strongly GV-W-cotorsion=strongly GV-E-cotorsion})}
	The following statements are equivalent for a domain $R$.
	\begin{enumerate}
		\item $R$ is a Pr\"{u}fer $v$-multiplication domain.
		\item The classes of all $w$-Enochs and $w$-Warfield cotorsion $R$-modules coincide.
		\item The classes of all $w$-torsion-free and $w$-flat $R$-modules coincide.
		\item The classes of all $\GV$-Enochs and $\GV$-Warfield cotorsion $R$-modules coincide.
	\end{enumerate}
\end{thm-b}

\begin{thm-c}{\rm (Theorem \ref{M-cotorsion=GV-E-cotorsion})}
	The following statements are equivalent for a domain $R$.
	\begin{enumerate}
		\item $R$ is almost perfect.
		\item The classes of all $w$-flat and $w$-strongly flat $R$-modules coincide.
		\item The classes of all Matlis and $\GV$-Enochs cotorsion $R$-modules coincide.
	\end{enumerate}
\end{thm-c}

\begin{thm-d}{\rm (Theorem \ref{dedekind domain})}
	The following statements are equivalent for a domain $R$.
	\begin{enumerate}
		\item $R$ is a Dedekind domain.
		\item The classes of all $\GV$-Warfield, $\GV$-Enochs and Matlis cotorsion $R$-modules coincide.
		\item The classes of all $w$-torsion-free, $w$-flat and $w$-strongly flat $R$-modules coincide.
	\end{enumerate}
\end{thm-d}

\section{Cotorsion and Tor pairs relative to $\tau_w$}

In 1966, Dickson \cite{Dickson66} introduced torsion theories for abelian categories by exploiting the $\Hom$-functor. Replacing formally the $\Hom$-functor by the $\Ext$-functor, cotorsion pairs (or cotorsion theories) were invented by Salce \cite{Salce79} in the category of abelian groups, and were rediscovered by Enochs et al. in the 1990’s.

For a class $\C$ of $R$-modules, define

\[\rule{0pt}{0pt}^{\bot}\C=\left\{M\in\m~|~\hbox{$\Ext_R^1(M,C)=0$ for all $C\in\C$}\right\},\]
\[\C^{\bot}=\left\{M\in\m~|~\hbox{$\Ext_R^1(C,M)=0$ for all $C\in\C$}\right\}.\]

Recall that a pair $(\A,\B)$ of $R$-modules is said to be a \textit{cotorsion pair} if both $\A=\rule{0pt}{0pt}^{\bot}\B$ and $\B=\A^{\bot}$ hold. Cotorsion pairs have been used to study covers and envelopes, particularly in the proof of the flat cover conjecture \cite{BBE01}. They have also been used in tilting theory and in the representation theory of Artin algebras.

$\Tor$-pairs were introduced as an analogue of cotorsion pairs, that is pairs of classes which are orthogonal relative to the $\Tor$-functor instead of the $\Ext$-functor. For a class $\C$ of $R$-modules, write

\[\C^{\top}=\left\{M\in\m~|~\hbox{$\Tor^R_1(M,C)=\Tor^R_1(C,M)=0$ for all $C\in\C$}\right\}.\]
A pair $(\A,\B)$ is called a \textit{$\Tor$-pair} if both $\A=\B^{\top}$ and $\B=\A^{\top}$.

For our purpose, we shall consider in this section the cotorsion and $\Tor$ pairs relative to the hereditary torsion theory $\tau_w$.

Let $\C$ be a class of $R$-modules. Then we define

\[\C^{\top_w}=\left\{M\in\m~|~\hbox{$\Tor^R_1(M,C)=\Tor^R_1(C,M)$ is $\GV$-torsion for all $C\in\C$}\right\},\]

\[\rule{0pt}{0pt}^{\bot_w}\C=\left\{M\in\m~|~\hbox{$\Ext_R^1(M,C)$ is $\GV$-torsion for all $C\in\C$}\right\},\]
and
\[\C^{\bot_w}=\left\{M\in\m~|~\hbox{$\Ext_R^1(C,M)$ is $\GV$-torsion for all $C\in\C$}\right\}.\]

\subsection{Tor-pairs relative to $\tau_w$}\quad

By replacing the vanishing of the $\Tor$-functor with the $\GV$-torsion of the $\Tor$-functor in the definition of $\Tor$-pairs, we introduce the following definition.

\begin{definition}\label{the definition of w-Tor-pairs}
	Let $\A$ and $\B$ be two classes of $R$-modules. The pair $(\A,\B)$ is called a \textit{$w$-$\Tor$-pair} if $\A=\B^{\top_w}$ and $\B=\A^{\top_w}$.
\end{definition}

It is worth noting that the notion of a $w$-$\Tor$-pair appeared first in \cite[Exercise 12.42]{WK24} under the name of a $w$-$\Tor$-torsion theory.

We use $\mathcal{FL}_w$ to denote the class of all $w$-flat $R$-modules.

\begin{example}\label{an example of a w-Tor-pair}
	$(\mathcal{FL}_w,\m)$ is a $w$-$\Tor$-pair.
\end{example}

Next, we discuss the relationship between the $w$-$\Tor$-pairs and the $\Tor$-pairs. 

\begin{proposition}\label{the relation between w-Tor-pair and Tor-pair}
	If $R$ is a DW-ring, then a pair $(\A,\B)$ is a $w$-$\Tor$-pair if and only if it is a $\Tor$-pair.
\end{proposition}

However, in general, a $\Tor$-pair is not necessarily a $w$-$\Tor$-pair, and a $w$-$\Tor$-pair is not necessarily a $\Tor$-pair, either.

\begin{remark}
	By \cite[Proposition 2.1]{TAB19}, it is not difficult to see that $(\mathcal{FL}_w,\m)$ is a $\Tor$-pair if and only if $\mathcal{FL}_w=\mathcal{FL}$, if and only if $(\mathcal{FL},\m)$ is a $w$-$\Tor$-pair, if and only if $R$ is a DW-ring, where $\mathcal{FL}$ denotes the class of all flat $R$-modules.
\end{remark}

Since $\Tor$ commutes with direct limits, the following proposition is a consequence of the fact that the class of all $\GV$-torsion $R$-modules is closed under direct limits.

\begin{proposition}\label{w-Tor pair and direct limits}
	If $(\A,\B)$ is a $w$-$\Tor$-pair, then both $\A$ and $\B$ are closed under direct limits.
\end{proposition}

\begin{proposition}\label{w-Tor pair}
	Let $\mathcal{S}$ be a class of $R$-modules. Then
	\begin{enumerate}
		\item $\mathcal{S}\subseteq \left(\mathcal{S}^{\top_{w}}\right)^{\top_{w}}$;
		\item $\left( \left(\mathcal{S}^{\top_{w}}\right)^{\top_{w}}\right)^{\top_{w}}=\mathcal{S}^{\top_{w}}$;
		\item $\left(\mathcal{S}^{\top_{w}}, \left(\mathcal{S}^{\top_{w}}\right)^{\top_{w}}\right)$ is a $w$-$\Tor$-pair.
	\end{enumerate}
\end{proposition}
\begin{proof}
	The proof is straightforward.
\end{proof}

\subsection{Cotorsion pairs relative to $\tau_w$}\quad

By analogy with the $w$-$\Tor$-pair, we define the $w$-cotorsion pair as follows:

\begin{definition}\label{the definition of w-cotorsion-pairs}
	Let $\A$ and $\B$ be two classes of $R$-modules. Then the pair $(\A,\B)$ is called a \textrm{$w$-cotorsion pair}, if $\A=\rule{0pt}{10pt}^{\bot_w}\hspace{-2pt}\B$ and $\B=\A^{\bot_w}$.
\end{definition}

Before giving some examples of $w$-cotorsion pairs we recall some definitions. If $M$ is an $R$-module and $s\in R$, then let $\eta^M_s:M\longrightarrow M$ denote the multiplication map $m\longmapsto sm$, $\forall m\in M$.

\begin{definition}{\rm (\cite[p. 3416, Definition]{WQ20})}
	\begin{enumerate}
		\item A short exact sequence of $R$-modules
		$$0\rightarrow A\stackrel{f}{\rightarrow} B\stackrel{g}{\rightarrow} C\rightarrow 0$$
		is said to be \textit{$w$-split} if there exist $J=\langle
		d_1,\dots,d_n\rangle\in\GV(R)$ and $h_1,\dots,h_n\in\Hom_R(C,B)$
		such that $\eta_{d_k}^C=gh_k$ for all $k=1,\dots,n$.
		\item An $R$-module $M$ is said to be \textit{$w$-split} if
		there is a $w$-split short exact sequence of $R$-modules
		\[0\rightarrow L\rightarrow P\rightarrow M\rightarrow 0\] 
		with $P$ projective.
	\end{enumerate}
\end{definition}

It was shown in \cite[Proposition 2.4]{WQ20} that an $R$-module $M$ is $w$-split if and only if $\Ext^1_R(M,N)$ is $\GV$-torsion for all $R$-modules $N$. 

Recall also from \cite{Wu24} that an $R$-module $N$ is called \textit{$iw$-split} if there exists a $w$-split short exact sequence of $R$-modules
\[0\rightarrow N\rightarrow E\rightarrow C\rightarrow 0\]
with $E$ injective. It was proved in \cite[Theorem 2.3]{Wu24} that an $R$-module $N$ is $iw$-split if and only if $\Ext^1_R(M,N)$ is $\GV$-torsion for all $R$-modules $M$. 

The class of all $w$-split modules and all $iw$-split modules is denoted by $\mathcal{SP}_w$ and $\mathcal{SP}_{iw}$, respectively.

Next, we collect two preliminary lemmas for $w$-split exact sequences.

\begin{lemma}{\rm (\cite[Proposition 6.7.4]{WK24})}\label{w-split exact sequence}
	Assume that $\xi:0\rightarrow A\rightarrow B\rightarrow C\rightarrow 0$ is an exact sequence of $R$-modules. If $\Ext^1_R(C,A)$ is a $\GV$-torsion module, then $\xi$ is a $w$-split exact sequence.
\end{lemma}

\begin{lemma}\label{localization of w-split exact sequence}
	Let $\xi:0\rightarrow A\xrightarrow[]{f} B\xrightarrow[]{g} C\rightarrow 0$ be a $w$-split exact sequence of $R$-modules. Then the exact sequence $\xi_{\p}:0\rightarrow A_{\p}\xrightarrow[]{f_{\p}} B_{\p}\xrightarrow[]{g_{\p}} C_{\p}\rightarrow 0$ is split over $R_{\p}$ for all prime $w$-ideals $\p$ of $R$.
\end{lemma}

\begin{proof}
	Let $\p$ be a prime $w$-ideal of $R$. Since $\xi$ is $w$-split, there exist a $\GV$-ideal $J=(d_1,d_2,\cdots,d_n)$ and $h_k\in\Hom_R(C,B)$ such that $gh_k=d_k\cdot \textbf{1}_C$, $k=1,2,\cdots,n$. As $J\nsubseteq\p$, we can choose some $t\in J$ with $t\notin\p$. Write $t=r_1d_1+r_2d_2+\cdots+r_nd_n$, $r_k\in R$, $k=1,2,\cdots,n$, and set $h=\frac{1}{t}\cdot\sum\limits_{k=1}^n\left(r_k\cdot h_k\right)_{\p}$. Then $h\in\Hom_{R_{\p}}(C_{\p},B_{\p})$, and for each $\frac{x}{s}\in C_{\p}$, $x\in C$, $s\in R\backslash\p$ we have
	\[g_{\p}h\left(\frac{x}{s}\right)=\frac{1}{t}\cdot\frac{g\sum\limits_{k=1}^n\left(r_k\cdot h_k\right)(x)}{s}=\frac{\sum\limits_{k=1}^nr_kgh_k(x)}{ts}=\frac{\sum\limits_{k=1}^nr_kd_kx}{ts}=\frac{tx}{ts}=\frac{x}{s}.\]
	Hence, $g_{\p}h=\textbf{1}_{C_{\p}}$, and so $\xi_{\p}$ is split over $R_{\p}$.
\end{proof}

\begin{proposition}\label{localizations of w-split modules}
	Let $M$ be a $w$-split $R$-module. Then, for each prime $w$-ideal $\mathfrak{p}$ of $R$, $M_{\mathfrak{p}}$ is free over $R_{\mathfrak{p}}$. 
\end{proposition}
\begin{proof}
	Let $\mathfrak{p}$ be any prime $w$-ideal of $R$. Since $M$ is $w$-split, there is a $w$-split exact sequence of $R$-modules $0\rightarrow L\rightarrow P\rightarrow M\rightarrow 0$ with $P$ projective. Therefore, by Lemma \ref{localization of w-split exact sequence}, we obtain a split exact sequence of $R_{\mathfrak{p}}$-modules $0\rightarrow L_{\mathfrak{p}}\rightarrow P_{\mathfrak{p}}\rightarrow M_{\mathfrak{p}}\rightarrow 0$ with $P_{\mathfrak{p}}$ a free $R_{\mathfrak{p}}$-module. Thus it follows that $M_{\mathfrak{p}}$ is also a free $R_{\mathfrak{p}}$-module.
\end{proof}

\begin{corollary}\label{w-split is w-flat}
	$\mathcal{SP}_w\subseteq\mathcal{FL}_w$, i.e., every $w$-split $R$-module is $w$-flat.
\end{corollary}
\begin{proof}
	This follows immediately from Proposition \ref{localizations of w-split modules} and the fact that an $R$-module $M$ is $w$-flat if and only if $M_{\mathfrak{m}}$ is a flat $R_{\mathfrak{m}}$-module for each maximal $w$-ideal $\mathfrak{m}$ of $R$.
\end{proof}

However, there exists an example of a $w$-flat module which is not $w$-split (\cite[Example 2.6]{WQ20}). It is then natural to ask the following: for which ring $R$ is it true that $\mathcal{SP}_w=\mathcal{FL}_w$? The answer to this question will be given in Theorem \ref{perfect ring}

\begin{example}\label{examples of w-cotorsion pairs}
	Both $(\mathcal{SP}_w,\m)$ and $(\m,\mathcal{SP}_{iw})$ are $w$-cotorsion pairs. 
\end{example}

It is easy to see the following.

\begin{proposition}\label{the relation between w-cotorsion pair and cotorsion pair}
	If $R$ is a DW-ring, then a pair $(\A,\B)$ is a $w$-cotorsion pair if and only if it is a cotorsion pair.
\end{proposition}

However, in general, a cotorsion pair is not necessarily a $w$-cotorsion pair, and a $w$-cotorsion pair is not necessarily a cotorsion pair, either.

\begin{remark}\quad
	\begin{enumerate}
		\item By using \cite[Theorem 2.7]{AA21}, it is easily seen that $(\mathcal{SP}_w,\m)$ is a cotorsion pair if and only if $\mathcal{SP}_w=\mathcal{P}_0$, if and only if $(\mathcal{P}_0,\m)$ is a $w$-cotorsion pair, if and only if $R$ is a DW-ring, where $\mathcal{P}_0$ denotes the class of all projective $R$-modules.
		\item It was proved in \cite[Theorem 2.17]{Wu24} that $R$ is a DW-ring if and only if every $iw$-split $R$-module is injective. Thus, we conclude that $(\m,\mathcal{SP}_{iw})$ is a cotorsion pair if and only if $\mathcal{SP}_{iw}=\mathcal{I}_0$, if and only if $(\m,\mathcal{I}_0)$ is a $w$-cotorsion pair, if and only if $R$ is a DW-ring, where $\mathcal{I}_0$ denotes the class of all injective $R$-modules.
	\end{enumerate}
\end{remark}

We next prove that a $w$-$\Tor$-pair induces both a cotorsion pair and a $w$-cotorsion pair.

\begin{proposition}\label{w-cotorsion pair}
	Let $\mathcal{S}$ be a class of $R$-modules. Then
	\begin{enumerate}
		\item $\mathcal{S}\subseteq\rule{0pt}{10pt}^{\bot_{w}}\left(\mathcal{S}^{\bot_{w}}\right)\bigcap\left(\rule{0pt}{10pt}^{\bot_{w}}\mathcal{S}\right)^{\bot_{w}}$.
		\item $\left(\rule{0pt}{10pt}^{\bot_{w}}\left(\mathcal{S}^{\bot_{w}}\right)\right)^{\bot_{w}}=\mathcal{S}^{\bot_{w}}$, and so $\left(\rule{0pt}{10pt}^{\bot_{w}}\left(\mathcal{S}^{\bot_{w}}\right),\mathcal{S}^{\bot_{w}}\right)$ is a $w$-cotorsion pair, called the $w$-cotorsion pair {\rm generated} by the class $\mathcal{S}$.
		\item $\rule{0pt}{20pt}^{\bot_{w}}\left(\left(\rule{0pt}{10pt}^{\bot_{w}}\mathcal{S}\right)^{\bot_{w}}\right)=\rule{0pt}{10pt}^{\bot_{w}}\mathcal{S}$, and hence $\left(\rule{0pt}{10pt}^{\bot_{w}}\mathcal{S},\left(\rule{0pt}{10pt}^{\bot_{w}}\mathcal{S}\right)^{\bot_{w}}\right)$ is a $w$-cotorsion pair, called the $w$-cotorsion pair {\rm cogenerated} by the class $\mathcal{S}$.
	\end{enumerate}
\end{proposition}
\begin{proof}
	The proof is a routine exercise. 
\end{proof}

Put $E_0:=\prod E(R/\mathfrak{m})$ with the product taken over all maximal $w$-ideals $\mathfrak{m}$ of $R$. Then since $R/\mathfrak{m}$ is $\GV$-torsion-free and the torsion theory $\tau_w$ is hereditary, $E_0$ is a $\GV$-torsion-free and injective $R$-module. For any $R$-module $M$, we write $M^{\dag}:=\Hom_R(M,E_0)$. Notice that $M^{\dag}$ is $\GV$-torsion-free for every $R$-module $M$. In fact, for each $J\in\GV(R)$,
\[\Hom_R(R/J,M^{\dag})\cong\left(M/JM\right)^{\dag}=0,\]
because $E_0$ is $\GV$-torsion-free.

\begin{lemma}\label{GV-torsion module}
	Let $M$ be an $R$-module. Then $M$ is a $\GV$-torsion module if and only if $M^{\dag}=0$.
\end{lemma}
\begin{proof}
	The necessity is clear. Conversely suppose that $M^{\dag}=0$. Then
	\[\prod\Hom_R(M,E(R/\mathfrak{m}))\cong M^{\dag}=0,\]
	and so $\Hom_R(M,E(R/\mathfrak{m}))=0$ for all $\mathfrak{m}\in\wmax(R)$. Thus, it follows from \cite[Chapter IX, Lemma 5.7]{FS01} that $M_{\mathfrak{m}}=0$ for all $\mathfrak{m}\in\wmax(R)$, i.e., $M$ is a $\GV$-torsion module.
\end{proof}

For the definition of complete cotorsion pairs, see \cite[Lemma 2.2.6]{GT06}. For the definition of closed cotorsion pairs and that of perfect cotorsion pairs, we refer to \cite[Definition 2.3.1]{GT06}.

\begin{proposition}\label{w-tor and cotor}
	Let $(\A,\B)$ be a $w$-$\Tor$-pair and $\C=\{B^\dag~|~B\in\B\}$. Then the following hold.
	\begin{enumerate}
		\item $\C\subseteq \A^\bot\subseteq \A^{\bot_w}$.
		\item $\A=\rule{0pt}{10pt}^{\bot}(\A^\bot)=\rule{0pt}{0pt}^{\bot}\C$. 
		\item $(\A,\A^{\bot})$ is a complete and closed cotorsion pair, and hence it is perfect.
		\item $\A=\rule{0pt}{10pt}^{\bot_w}(\A^{\bot_w})=\rule{0pt}{0pt}^{\bot_w}\C$, and so $(\A,\A^{\bot_w})$ is a $w$-cotorsion pair.
	\end{enumerate}
\end{proposition}  
\begin{proof} (1) Obviously, $\A^\bot\subseteq \A^{\bot_w}$. Let $B^\dag\in\C$ with $B\in\B$. Then for each $A\in\A$, $\Tor^R_1(A,B)$ is a $\GV$-torsion module. From this, we have
		\[\Ext^1_R(A,B^\dag)\cong\Tor^R_1(A,B)^\dag=0,\]
which yields $B^\dag\in \A^\bot$. Thus $\C\subseteq \A^\bot$.
	
(2) By (1),
		\[\A\subseteq\rule{0pt}{10pt}^{\bot}(\A^{\bot})\subseteq\rule{0pt}{0pt}^{\bot}\C.\]
Conversely, let $A\in\rule{0pt}{0pt}^{\bot}\C$. Then for any $B\in\B$, we have $B^\dag\in\C$ and
		\[\Tor^R_1(A,B)^\dag\cong\Ext^1_R(A,B^\dag)=0,\]
and consequently $\Tor^R_1(A,B)$ is $\GV$-torsion by Lemma \ref{GV-torsion module}. Hence $A\in\B^{\top_w}=\A$, and so $\rule{0pt}{0pt}^{\bot}\C\subseteq\A$. 
	
(3) By (2), $(\A,\A^{\bot})$ is a cotorsion pair cogenerated by the class $\C$. Since by \cite[Lemma 2.1]{E84} every element of $\C$ is pure injective, it follows from \cite[Theorem 3.2.9]{GT06} that $(\A,\A^{\bot})$ is complete, closed, and hence perfect.
	
(4) First note that
	\[\A\subseteq\rule{0pt}{10pt}^{\bot_w}(\A^{\bot_w})\subseteq\rule{0pt}{0pt}^{\bot_w}\C.\]
Conversely, let $A\in\rule{0pt}{0pt}^{\bot_w}\C$. Then for each $B\in\B$, we have that $B^\dag\in\C$ and \[\Tor^R_1(A,B)^\dag\cong\Ext^1_R(A,B^\dag)\]
is $\GV$-torsion. But $\Tor^R_1(A,B)^\dag$ is also a $\GV$-torsion-free module, hence we obtain $\Tor^R_1(A,B)^\dag=0$. Consequently, by Lemma \ref{GV-torsion module}, $\Tor^R_1(A,B)$ is a $\GV$-torsion module, which implies $A\in\B^{\top_w}=\A$. Thus $\rule{0pt}{0pt}^{\bot_w}\C\subseteq\A$, and the first part of (4) holds. Finally, by Proposition \ref{w-cotorsion pair}(3), $(\A,\A^{\bot_w})$ is a $w$-cotorsion pair cogenerated by the class $\C$.
\end{proof}

\begin{proposition}\label{hereditary w-cotorsion pair}
	Let $\mathcal{C}_{w}=(\A,\B)$ be a $w$-cotorsion pair. Then the following statements are equivalent.
	\begin{enumerate}
		\item $\A$ is resolving.
		\item $\B$ is coresolving.
		\item $\Ext^i_R(A,B)$ is $\GV$-torsion for all $i\geqslant 1$, $A\in\A,B\in\B$.
	\end{enumerate}
	In this case, the $w$-cotorsion pair $\mathcal{C}_w$ is called {\rm hereditary}.
\end{proposition}
\begin{proof}
	The proof of this lemma is similar to that of \cite[Lemma 2.2.10]{GT06}.
\end{proof}

\section{Enochs-cotorsion modules relative to $\tau_w$}

In this section, we shall consider two types of Enochs-cotorsion modules relative to $\tau_w$.

\subsection{$w$-Enochs-cotorsion modules}\quad

In \cite{ATA18}, Almahdi et al. introduced the concept of $w$-Enochs-cotorsion modules. 

\begin{definition}{\rm (\cite[Definition 2.1]{ATA18})}\label{the definition of w-E-cotorsion}
	Let $M$ be an $R$-module. Then $M$ is called a \textit{$w$-Enochs-cotorsion module}, if $\Ext^1_R(F,M)=0$ for any $w$-flat $R$-module $F$, that is, if $M\in{\mathcal{FL}_w}^\bot$.
\end{definition} 

Denote by $\mathcal{EC}_w$ the class of all $w$-Enochs-cotorsion $R$-modules. Then $\mathcal{EC}_w={\mathcal{FL}_w}^\bot$. 

Notice that $(\mathcal{FL}_w,\m)$ is a $w$-$\Tor$-pair, Proposition \ref{w-tor and cotor}(3) gives the following result.

\begin{proposition}{\rm (\cite[Theorem 3.5]{Z19})}\label{w-E-cotorsion pair}
	$(\mathcal{FL}_w,\mathcal{EC}_w)$ is a hereditary and perfect cotorsion pair. Hence, every $R$-module has a $w$-flat cover.
\end{proposition}

Let $\mathcal{EC}$ denote the class of all Enochs-cotorsion $R$-modules. Then $\mathcal{EC}_w\subseteq\mathcal{EC}$, i.e., every $w$-Enochs-cotorsion module is Enochs-cotorsion. However, the converse is not true in general. In fact, we have:

\begin{proposition}{\rm (\cite[Proposition 2.4]{ATA18})}\label{w-E-cotorsion=E-cotorsion}
	The following statements are equivalent for a ring $R$.
	\begin{enumerate}
		\item $\mathcal{EC}_w=\mathcal{EC}$.
		\item $\mathcal{FL}_w=\mathcal{FL}$.
		\item $R$ is a DW-ring.
	\end{enumerate}
\end{proposition}
\begin{proof} $(1)\Rightarrow (2)$ By Proposition \ref{w-E-cotorsion pair}, we obtain
	\[\mathcal{FL}_w=~^{\bot}\mathcal{EC}_w=~^{\bot}\mathcal{EC}=\mathcal{FL}.\]
	
	$(2)\Rightarrow (3)$ See the proof of \cite[Proposition 2.4]{ATA18}.
	
	$(3)\Rightarrow (1)$ This is trivial.
\end{proof}

\subsection{$\GV$-Enochs-cotorsion modules}\quad

In this subsection, we introduce and study a new type of Enochs-cotorsion modules relative to $\tau_w$.

\begin{definition}\label{the definition of GV-E-cotorsion}
	Let $M$ be an $R$-module. Then $M$ is said to be a \textit{$\GV$-Enochs-cotorsion module} if $\Ext^1_R(F,M)$ is a $\GV$-torsion module for any $w$-flat $R$-module $F$, i.e., if $M\in{\mathcal{FL}_w}^{\bot_w}$.
\end{definition}

We will denote by $\mathcal{EC}_{\GV}$ the class of all $\GV$-Enochs-cotorsion $R$-modules, that is, $\mathcal{EC}_{\GV}={\mathcal{FL}_w}^{\bot_w}$. Clearly, $\mathcal{EC}_w\subseteq\mathcal{EC}_{\GV}$ and $\mathcal{SP}_{iw}\subseteq\mathcal{EC}_{\GV}$, i.e., both $w$-Enochs-cotorsion modules and $iw$-split modules are $\GV$-Enochs modules. 

Note that $(\mathcal{FL}_w,\m)$ is a $w$-$\Tor$-pair. Therefore, the following result is a consequence of Proposition \ref{w-tor and cotor}(4).

\begin{proposition}\label{GV-E-w-cotorsion pair}
	$\left(\mathcal{FL}_w,\mathcal{EC}_{\GV}\right)$ is a $w$-cotorsion pair.
\end{proposition}

Recall that an $R$-module $M$ is called a \textit{Matlis-cotorsion module} if $\Ext^1_R(Q,M)=0$, i.e., if $M\in Q^{\bot}$, where $Q$ is the classical ring of quotients of $R$. Let $\mathcal{MC}$ denote the class of all Matlis-cotorsion $R$-modules. Then $\mathcal{MC}=Q^{\bot}$.

\begin{proposition}\label{GV-E-cotorsion is M-cotorsion}
	If $R$ is a domain, then $\mathcal{MC}=Q^{\bot_{w}}$. Hence, $\mathcal{EC}_{\GV}\subseteq\mathcal{MC}$, i.e., every $\GV$-Enochs-cotorsion $R$-module is a Matlis-cotorsion module.
\end{proposition}
\begin{proof}
	Obviously, $\mathcal{MC}\subseteq Q^{\bot_{w}}$. Conversely, let $M\in Q^{\bot_{w}}$. Then $\Ext^1_R(Q,M)$ is a $\GV$-torsion $R$-module. But note that the $Q$-module $\Ext^1_R(Q,M)$ as an $R$-module is $\GV$-torsion-free. Therefore, $\Ext^1_R(Q,M)=0$, that is, $M\in\mathcal{MC}$, and so $Q^{\bot_{w}}\subseteq\mathcal{MC}$.
	
	Finally, since $Q$ is flat (and hence $w$-flat) over $R$, we obtain $\mathcal{EC}_{\GV}\subseteq\mathcal{MC}$.
\end{proof}

We will show, in Section 5, that $\mathcal{EC}_{\GV}\subsetneq\mathcal{MC}$ in general. Next, we shall see that the class of $\GV$-Enochs-cotorsion modules is useful in characterizing rings.

Recall that a ring $R$ is said to be a \textit{perfect ring} if every $R$-module has a projective cover; equivalently, if all flat $R$-modules are projective, if and only if every $R$-module is Enochs-cotorsion, and if and only if every flat $R$-module is Enochs-cotorsion. 

The perfect rings can now be also characterized in the following way.

\begin{theorem}\label{perfect ring}
	The following statements are equivalent for a ring $R$.
	\begin{enumerate}
		\item Every $R$-module is $\GV$-Enochs-cotorsion.
		\item Every $w$-flat $R$-module is $\GV$-Enochs-cotorsion;
		\item $\mathcal{FL}_w=\mathcal{SP}_w$.
		\item $R$ is a perfect ring.
	\end{enumerate}
\end{theorem}

\begin{proof}
	$(1)\Rightarrow (2)$ This is trivial.
	
	$(2)\Rightarrow (3)$ By Corollary \ref{w-split is w-flat}, we need only show that $\mathcal{FL}_w\subseteq\mathcal{SP}_w$. Let $F$ be any $w$-flat $R$-module. Then there is an exact sequence of $R$-modules
	\[0\rightarrow G\rightarrow P\rightarrow F\rightarrow 0,\eqno{(\star)}\]
	where $P$ is projective and $G$ is $w$-flat. By (2), $G$ is a $\GV$-Enochs-cotorsion module, and so $\Ext_R^1(F,G)$ is $\GV$-torsion. Thus, Lemma \ref{w-split exact sequence} says that $(\star)$ is a $w$-split exact sequence. Hence, $F$ is a $w$-split module.
	
	$(3)\Rightarrow (4)$ It suffices to prove that $R$ is a DW-ring. For this, we put
	\[M=\bigoplus\{R/J~|~J\in\GV(R)\}.\]
	Clearly, $M$ is a $\GV$-torsion $R$-module, and so it is a $w$-flat $R$-module. Therefore, by (3), $M$ is $w$-split over $R$. Thus, \cite[Lemma 2.3]{WQ20} says that there exists $J_0\in\GV(R)$ with $J_0M=0$, which yields $J_0\subseteq J$ for all $J\in\GV(R)$. In particular, $J_0\subseteq J_0^2$, and consequently $J_0=J_0^2$. This means that $J_0$ can be generated by an idempotent, and hence it is projective. But every projective module is a $w$-module, so $J_0=(J_0)_w=R$. This clearly forces $\GV(R)=\{R\}$, that is, $R$ is a DW-ring.
	
	$(4)\Rightarrow (1)$ Suppose that $R$ is a perfect ring. Then $\dim(R)=0$. Therefore, $R$ is a DW-ring, and $\mathcal{EC}_{\GV}=\mathcal{EC}$. Thus, (1) follows from \cite[Proposition 3.3.1]{X96}.
\end{proof}

It is well known that a ring $R$ is a von Neumann regular ring if and only if every $R$-module is flat. Thus it is immediate that $R$ is von Neumann regular if and only if every Enochs-cotorsion $R$-module is injective. In fact, this ring can also be characterized by the coincidence of the classes $\mathcal{EC}_{\GV}$ and $\mathcal{SP}_{iw}$.

\begin{theorem}\label{strong GV-E-cotorsion=iw-split}
	The following statements are equivalent for a ring $R$.
	\begin{enumerate}
		\item $\mathcal{EC}_{\GV}=\mathcal{I}_0$.
		\item $\mathcal{EC}_{\GV}=\mathcal{SP}_{iw}$.
		\item $\mathcal{FL}_w=\m$, i.e., every $R$-module is $w$-flat.
		\item $\mathcal{EC}_{w}=\mathcal{I}_0$.
		\item $R$ is a von Neumann regular ring.
	\end{enumerate}
\end{theorem}

\begin{proof}
	$(1)\Rightarrow (2)$ This is trivial.
	
	$(2)\Rightarrow (3)$ From Proposition\ref{GV-E-w-cotorsion pair}, we have
	\[\mathcal{FL}_w=\rule{0pt}{10pt}^{\bot_w}\mathcal{EC}_{\GV}=\rule{0pt}{10pt}^{\bot_w}\mathcal{SP}_{iw}=\m.\]
	
	$(3)\Rightarrow (4)$ Clear.
	
	$(4)\Rightarrow (5)$ see \cite[Corollary 2.9]{ATA18}.
	
	$(5)\Rightarrow (1)$ It follows from the fact that every von Neumann regular ring is a DW-ring.
\end{proof}

We close this section proving that the coincidence of the classes of two $\tau_w$-relative versions of Enochs-cotorsion modules characterizes DW-rings.

\begin{theorem}\label{GV-E-torsion=w-E-torsion}
	The following statements are equivalent for a ring $R$.
	\begin{enumerate}
		\item $\mathcal{EC}_{\GV}=\mathcal{EC}_w$.
		\item every $iw$-split $R$-module is $w$-Enochs-cotorsion.
		\item $R$ is a DW-ring.
	\end{enumerate}
\end{theorem}

\begin{proof}
	$(1)\Rightarrow (2)$ and $(3)\Rightarrow (1)$ are both obvious.
	
	$(2)\Rightarrow (3)$ Let $J\in\GV(R)$ and write $N=J/J^2$. Then $JN=0$, and consequently $N$ is $iw$-split by \cite[Theorem 2.10]{Wu24}. From (2), $N$ is a $w$-Enochs-cotorsion module. Now we consider the following exact sequence of $R$-modules
	\[0\rightarrow\Hom_R(R/J,N)\rightarrow\Hom_R(R,N)\rightarrow\Hom_R(J,N)\rightarrow\Ext^1_R(R/J,N).\]
	Since $R/J$ is a $\GV$-torsion $R$-module, it is $w$-flat over $R$, and so $\Ext^1_R(R/J,N)=0$. Let $f:J\rightarrow N$ be the natural epimorphism, that is, $f(a)=\bar{a}$, for each $a\in J$. Then there exists a homomorphism of $R$-modules $g:R\rightarrow N$ such that $g(a)=f(a)$ for all $a\in J$, i.e., $g|_J=f$. Set $\bar{b}=g(1)$, $b\in J$. Therefore, \[g(1)=\bar{b}=f(b)=g(b)=bg(1)=b\bar{b}=\bbar{b^2}=\bar{0},\]
	and hence $g=0$. As $f$ is an epimorphism, so is $g$. Thus $N=0$, i.e., $J=J^2$. Now, the rest of the proof is the same as that in $(3)\Rightarrow (4)$ of Theorem \ref{perfect ring}.
\end{proof}

\section{Warfield-cotorsion modules relative to $\tau_w$}

In this section, we shall introduce and study two classes of Warfield-cotorsion modules relative to $\tau_w$. For this purpose we first have to discuss the torsion-free modules relative to $\tau_w$.

\subsection{Torsion-free modules relative to $\tau_w$}\quad

The notation $R^{\times}$ is used for the set of nonzero divisors of $R$. Recall that an $R$-module $M$ is said to be \textit{torsion-free}, if $sx=0$ implies $x=0$ for all $s\in R^{\times}$ and all $x\in M$, equivalently, if $\Tor^R_1(R/sR,M)=0$ for all $s\in R^{\times}$. We will denote by $\mathcal{TF}$ the class of all torsion-free $R$-modules.

\begin{definition}\label{the definition of w-torsion-free}
	Let $M$ be an $R$-module. Then $M$ is called \textit{$w$-torsion-free} if $\Tor^R_1(R/sR,M)$ is a $\GV$-torsion $R$-module for all $s\in R^{\times}$.
\end{definition}

Let us denote by $\mathcal{TF}_w$ the class of all $w$-torsion-free $R$-modules. Clearly, $\mathcal{TF}\subseteq \mathcal{TF}_w$, i.e., every torsion-free module is $w$-torsion-free. Also, $\mathcal{FL}_w\subseteq\mathcal{TF}_w$, i.e., every $w$-flat module is $w$-torsion-free. In particular, every $\GV$-torsion module is $w$-torsion-free. However, a $w$-torsion-free module is not necessarily a torsion-free module. 

\begin{example}
	Let $R$ be a domain which is not a DW-domain. Then there exists $J\in\GV(R)$ with $0\neq J\neq R$. Thus, $R/J$ is a $\GV$-torsion $R$-module, and hence it is $w$-torsion-free over $R$. But it is obvious that $R/J$ is not a torsion-free $R$-module.
\end{example}

The next two propositions gives a number of characterizations of $w$-torsion-free modules.

\begin{proposition}\label{some characterizations of w-torsion-free modules}
	Let $M$ be an $R$-module. Then the following statements are equivalent.
	\begin{enumerate}
		\item $M$ is $w$-torsion-free.
		\item For any $x\in M$ and any $s\in R^\times$, if $sx=0$, then there is $J\in\GV(R)$ with $Jx=0$.
		\item The natural $R$-homomorphism $\theta:M\longrightarrow Q\bigotimes_R M$ is a $w$-monomorphism.
		\item  $\Tor^R_1(K,M)$ is a $\GV$-torsion $R$-module, where $K=Q/R$.
	\end{enumerate}
\end{proposition}
\begin{proof}
	$(1)\Leftrightarrow (2)$ This follows from the isomorphism
		\[\Tor^R_1(R/sR,M)\cong\{x\in M~|~sx=0\},\]
		where $s\in R^{\times}$.
		
	$(2)\Leftrightarrow (3)$ Since $\ker\theta=\{x\in M~|~\mbox{$sx=0$ for some $s\in R^{\times}$}\}$, the equivalence of (2) and (3) is clear.
	
	$(3)\Leftrightarrow (4)$ It follows immediately from the isomorphism
	\[\Tor^R_1(K,M)\cong\ker\theta.\]
\end{proof}

By Proposition \ref{some characterizations of w-torsion-free modules}, it is not difficult to check that $\mathcal{TF}_w$ is closed under extensions and $w$-isomorphisms. Furthermore, if $R$ is a domain, then one can easily get the following result.

\begin{proposition}\label{some characterizations of w-torsion-free modules over domains}
	Let $R$ be a domain and $M$ an $R$-module. Then the following statements are equivalent.
	\begin{enumerate}
		\item $M$ is $w$-torsion-free.
		\item  $M_{\mathfrak{p}}$ is a torsion-free $R_{\mathfrak{p}}$-module for any prime $w$-ideal $\mathfrak{p}$ of $R$.
		\item  $M_{\mathfrak{m}}$ is a torsion-free $R_{\mathfrak{m}}$-module for any maximal $w$-ideal of $R$.
	\end{enumerate}
\end{proposition}

\begin{corollary}\label{w-torsion-free and GV-torsion-free}
	Let $M$ be a $GV$-torsion-free and $w$-torsion-free $R$-module. Then $M$ is torsion-free.
\end{corollary}

\begin{proof}
	For any $x\in M$ and any $x\in R^\times$ with $sx=0$, we obtain, via Proposition \ref{some characterizations of w-torsion-free modules}(2), that there exists $J\in\GV(R)$ such that $Jx=0$. But $M$ is $\GV$-torsion-free, so $x=0$, which gives that $M$ is torsion-free. 
\end{proof}

It is obvious that if $R$ is a DW-ring, then $\mathcal{TF}_w=\mathcal{TF}$. The converse is true in the domain case.

\begin{proposition}\label{w-torsion-free=torsion-free}
	If $R$ satisfies $\mathcal{TF}_w=\mathcal{TF}$, then every proper $\GV$-ideal is contained in the set of zero-divisors of $R$. Hence, if $R$ is a domain, then it is a DW-domain.
\end{proposition}

\begin{proof}
	Let $J$ be any proper $\GV$-ideal of $R$. Then $R/J$ is a nonzero $\GV$-torsion $R$-module. Therefore, $R/J\in\mathcal{TF}_w=\mathcal{TF}$, that is, $R/J$ is a torsion-free $R$-module. Notice that for each $a\in J$, $a\cdot\bar{1}=0$, where $\bar{1}=1+J$. Thus, if $a$ is a nonzero divisor, then $\bar{1}=0$, i.e., $1\in J$, and so $J=R$, a contradiction.
	
	Furthermore, if $R$ is a domain, then $\GV(R)=\{R\}$, and hence $R$ is a DW-domain.
\end{proof}

To close this subsection, we give a characterization of Pr\"{u}fer $v$-multiplication domains in terms of $w$-torsion-free modules. Recall that a domain $R$ is a \textit{Pr\"{u}fer $v$-multiplication domain} (P$v$MD) if the monoid of finite-type $v$-ideals of $R$ forms a group under $v$-multiplication, equivalently, if $R_{\mathfrak{m}}$ is a valuation domain for each $\mathfrak{m}\in\wmax(R)$ (see, for example, \cite[Theorem 7.5.9]{WK24}). The notion of P$v$MDs comes from multiplicative ideal theory, and
various ideal-theoretic properties of them have been considered by
many authors. Although these rings were studied by Krull, more
recent interest in them was sparked by Griffin's paper \cite{Griffin67}.
Examples of P$v$MDs are Pr\"ufer domains, Krull domains, GCD
domains, integrally closed coherent domains, etc.

\begin{proposition}\label{w-torsion-free=w-flat}
	Let $R$ be a domain. Then $R$ is a P$v$MD if and only if $\mathcal{TF}_w=\mathcal{FL}_w$.
\end{proposition}
\begin{proof}
	Suppose that $R$ is a P$v$MD and let $M\in\mathcal{TF}_w$. Then for each maximal $w$-ideal $\mathfrak{m}$ of $R$, Proposition \ref{some characterizations of w-torsion-free modules over domains} says that $M_{\mathfrak{m}}$ is torsion-free over $R_{\mathfrak{m}}$.  Since $R_{\mathfrak{m}}$ is a valuation domain, $M_{\mathfrak{m}}$ is flat over $R_{\mathfrak{m}}$, and so $M\in\mathcal{FL}_w$. Thus, $\mathcal{TF}_w\subseteq\mathcal{FL}_w$, and hence $\mathcal{TF}_w=\mathcal{FL}_w$.
	
	Conversely, assume that $\mathcal{TF}_w=\mathcal{FL}_w$. Then $\mathcal{TF}\subseteq\mathcal{FL}_w$. Consequently, by \cite[Theorem 3.5]{WQ15}, $R$ is a P$v$MD.
\end{proof}

\subsection{$w$-Warfield-cotorsion modules}\quad

Now a modification of the definition of Warfield-cotorsion modules by replacing ``torsion-free'' by ``$w$-torsion-free'' gives the following definition.

\begin{definition}\label{the definition of w-W-cotorsion}
	Let $M$ be an $R$-module. Then $M$ is called a \textit{$w$-Warfield-cotorsion module} if $\Ext^1_R(F,M)=0$ for any $w$-torsion-free $R$-module $F$, i.e., if $M\in{\mathcal{TF}_w}^{\bot}$.
\end{definition}

Let $\mathcal{WC}_w$ denote the class of all $w$-Warfield-cotorsion $R$-modules. Then $\mathcal{WC}_w={\mathcal{TF}_w}^{\bot}$. Of course, $\mathcal{WC}_w\subseteq\mathcal{WC}$ and $\mathcal{WC}_w\subseteq\mathcal{EC}_w$.

Note that $\mathcal{TF}_w=\mathcal{S}^{\bot_w}$, where $\mathcal{S}=\{R/sR~|~s\in R^\times\}$. Hence, by Proposition \ref{w-Tor pair}(3), $\left(\mathcal{TF}_w,{\mathcal{TF}_w}^{\top_w}\right)$ is a $w$-$\Tor$-pair. Thus, the following result follows immediately from Proposition \ref{w-tor and cotor}(3).

\begin{proposition}\label{w-W-cotorsion pair}
	$\left(\mathcal{TF}_w,\mathcal{WC}_w\right)$ a complete and closed cotorsion pair, and hence it is perfect.
\end{proposition}

Next, we characterize rings over which the classes $\mathcal{WC}_w$ and $\mathcal{WC}$ (respectively, $\mathcal{WC}_w$ and $\mathcal{EC}_w$) coincide.

\begin{proposition}\label{w-W-cotrosion=W-cotorsion}
	The following statements are equivalent for a ring $R$.
	\begin{enumerate}
		\item $\mathcal{WC}_w=\mathcal{WC}$.
		\item $\mathcal{TF}_w=\mathcal{TF}$.
	\end{enumerate}
	Furthermore, if $R$ is a domain, then each of the above statements is equivalent to 
	\begin{enumerate}
		\setcounter{enumi}{2}
		\item $R$ is a DW-domain.
	\end{enumerate}
\end{proposition}

\begin{proof}
	By Proposition \ref{w-torsion-free=torsion-free}, it suffices to prove the equivalence of (1) and (2). But this is trivial, by comparing the cotorsion pairs.
\end{proof}

\begin{proposition}\label{w-W-cotorsion=w-E-cotorsion}
	The following statements are equivalent for a ring $R$.
	\begin{enumerate}
		\item $\mathcal{WC}_w=\mathcal{EC}_w$.
		\item $\mathcal{TF}_w=\mathcal{FL}_w$
	\end{enumerate}
	Furthermore, if $R$ is a domain, then each of the above statements is equivalent to 
	\begin{enumerate}
		\setcounter{enumi}{2}
		\item $R$ is a P$v$MD.
	\end{enumerate}
\end{proposition}

\begin{proof}
	To apply Proposition \ref{w-torsion-free=w-flat} we need only show that (1) and (2) are equivalent. However, by comparing the cotorsion pairs, it is clear.
\end{proof}

\subsection{$\GV$-Warfield-cotorsion modules}\quad

We now consider another class of Warfield-cotorsion modules relative to $\tau_w$.

\begin{definition}\label{the definition of GV-W-cotorsion}
	Let $M$ be an $R$-module. Then $M$ is said to be a \textit{$\GV$-Warfield-cotorsion module} if $\Ext^1_R(F,M)$ is a $\GV$-torsion module for any $w$-torsion-free $R$-module $F$, i.e., if $M\in{\mathcal{TF}_w}^{\bot_w}$.
\end{definition}

We shall denote by $\mathcal{WC}_{\GV}$ the class of all $\GV$-Warfield-cotorsion $R$-modules. Then $\mathcal{WC}_{\GV}={\mathcal{TF}_w}^{\bot_w}$. Obviously, $\mathcal{WC}_{w}\subseteq \mathcal{WC}_{\GV}$ and $\mathcal{SP}_{iw}\subseteq \mathcal{WC}_{\GV}$. 

\begin{proposition}\label{GV-W-w-cotorsion pair}
	$\left(\mathcal{TF}_w,\mathcal{WC}_{\GV}\right)$ is a $w$-cotorsion pair.
\end{proposition}

\begin{proof}
	Since $\left(\mathcal{TF}_w,{\mathcal{TF}_w}^{\top_w}\right)$ is a $w$-$\Tor$-pair, the proof follows immediately Proposition \ref{w-tor and cotor}(4).
\end{proof}

Warfield \cite{Warfield70} characterized the Warfield-cotorsion modules over domains as Matlis-cotorsion modules of injective dimension $\leqslant 1$; see for example \cite[Chapter XIII, Theorem 8.3]{FS01}. The next proposition is the $w$-theoretic analogue of this characterization. For this we recall the $iw$-split dimension of modules from \cite{Wu24}. 

Let $N$ be an $R$-module. Then $\iwsd_RN\leqslant n$ ($\iwsd$ abbreviates \textit{$iw$-split dimension}) if there is an exact sequence (called a finite \textit{$iw$-split coresolution}) of $R$-modules
	\[0\rightarrow N\rightarrow E_0\rightarrow E_1\rightarrow\cdots\rightarrow E_n\rightarrow 0,\]
where each $E_i$ is $iw$-split, $i=1,2,\cdots,n$. If no such finite $iw$-split coresolution exists, then $\iwsd_RN=\infty$; otherwise, $\iwsd_RN=n$ if $n$ is the length of a shortest $iw$-split coresolution of $N$. It is shown that $\iwsd_RN\leqslant n$ if and only if $\Ext_R^{n+1}(M,N)$ is a $\GV$-torsion module for all $R$-modules $M$. The \textit{$w$-split dimension}, $\wsd$, is defined dually using \textit{$w$-split resolutions}.

\begin{proposition}\label{gv-w-cotor}
	Let $R$ be a domain. If $M$ is a $\GV$-Warfield-cotorsion $R$-module, then it is a Matlis-cotorsion module with $\iwsd_RM\leqslant 1$. The converse holds provided $M$ is a $w$-module.
\end{proposition}
\begin{proof}
	If $M$ is a $\GV$-Warfield-cotorsion $R$-module, then it is $\GV$-Enochs-cotorsion and hence a Matlis-cotorsion module by Proposition \ref{GV-E-cotorsion is M-cotorsion}. Now, for any $R$-module $N$, consider the exact sequence of $R$-modules $0\rightarrow L\rightarrow F\rightarrow N\rightarrow 0$ with $F$ free. Then $\Ext^2_R(N,M)\cong\Ext^1_R(L,M)$ is a $\GV$-torsion module, since $L\in\mathcal{TF}\subseteq\mathcal{TF}_w$. So $\iwsd_RM\leqslant 1$.
	
	Conversely, suppose $M$ is a Matlis-cotorsion $w$-module with $\iwsd_RM\leqslant 1$ and let $X$ be any $w$-torsion-free $R$-module. 
	
	First we prove that $\Ext^1_R(X,M)$ is a $\GV$-torsion module in case $X$ is a $\GV$-torsion-free. In fact, by Corollary \ref{w-torsion-free and GV-torsion-free}, $X$ is torsion-free. Thus, the injective envelop $E$ of $X$ is isomorphic to a direct sum of copies of $Q$. Applying $\Hom_R(-,M)$ to the exact sequence $0\rightarrow X\rightarrow E\rightarrow E/X\rightarrow 0$, we get the sequence
		\[0=\Ext^1_R(E,M)\rightarrow\Ext^1_R(X,M)\rightarrow\Ext^2_R(E/X,M)\]
	is exact, where the first Ext vanishes because $M$ is Matlis-cotorsion and the last Ext is $\GV$-torsion by $\iwsd_RM\leqslant 1$. Consequently, $\Ext^1_R(X,M)$ is $\GV$-torsion.
	
	In the general case, write $T=\tor(X)$ and consider the exact sequence of $0\rightarrow T\rightarrow X\rightarrow X/T\rightarrow 0$, where $X/T$ is $\GV$-torsion-free. We drive the exact sequence
	\[\Ext^1_R(X/T,M)\rightarrow\Ext^1_R(X,M)\rightarrow\Ext^1_R(T,M)=0,\]
	where the last Ext vanishes since $M$ is a $w$-module and since $T$ is $\GV$-torsion. 
	Notice that $X/T$ is also $w$-torsion-free by the fact that $\mathcal{TF}_w$ is closed under $w$-isomorphisms. Hence, the $\GV$-torsion-free case above gives that $\Ext^1_R(X/T,M)$ is $\GV$-torsion, and so is therefore $\Ext^1_R(X,M)$.

	Thus, it follows that $M$ is a $\GV$-Warfield-cotorsion module.
\end{proof}

The following two results characterize the rings over which $\mathcal{WC}_{w}=\mathcal{WC}_{\GV}$ (respectively, $\mathcal{SP}_{iw}=\mathcal{WC}_{\GV}$). 

\begin{proposition}\label{GV-W-torsion=w-W-torsion}
	The following statements are equivalent for a ring $R$.
	\begin{enumerate}
		\item $\mathcal{WC}_{\GV}=\mathcal{WC}_w$.
		\item $\mathcal{SP}_{iw}\subseteq\mathcal{WC}_w$, i.e., every $iw$-split $R$-module is $w$-Warfield-cotorsion.
        \item $R$ is a DW-ring.
	\end{enumerate}
\end{proposition}
\begin{proof}
	$(1)\Rightarrow (2)$ This is clear, as $\mathcal{SP}_{iw}\subseteq\mathcal{WC}_{\GV}$.
	
	$(2)\Rightarrow (3)$ Since $\mathcal{WC}_w\subseteq \mathcal{EC}_w$, it follows from Theorem \ref{GV-E-torsion=w-E-torsion} that $R$ is a DW-ring.
	
	$(3)\Rightarrow (1)$ This is obvious.
\end{proof}

\begin{theorem}\label{GV-W-cotorsion=iw-split}
	The following statements are equivalent for a ring $R$.
	\begin{enumerate}
		\item $\mathcal{WC}_{\GV}=\mathcal{SP}_{iw}$.
		\item $\mathcal{TF}_w=\m$, i.e., every $R$-module is $w$-torsion-free.
		\item $\mathcal{WC}_{w}=\mathcal{I}_0$.
	\end{enumerate}
		Furthermore, if $R$ is a domain, then each of the above statements is equivalent to 
	\begin{enumerate}
		\setcounter{enumi}{3}
		\item $R$ is a field.
	\end{enumerate}
\end{theorem}

\begin{proof}
	$(1)\Rightarrow (2)$ By Proposition \ref{GV-W-w-cotorsion pair}, we obtain
	\[\mathcal{TF}_w=\rule{0pt}{10pt}^{\bot_w}\mathcal{WC}_{\GV}=\rule{0pt}{10pt}^{\bot_w}\mathcal{SP}_{iw}=\m.\]
	
	$(2)\Rightarrow (1)$ and $(2)\Rightarrow (3)$ are both clear.
	
	$(3)\Rightarrow (2)$ By Proposition \ref{w-W-cotorsion pair}, we have
	\[\mathcal{TF}_w=\rule{0pt}{10pt}^{\bot}\mathcal{WC}_{w}=\rule{0pt}{10pt}^{\bot}\mathcal{I}_{0}=\m.\]
	
	Now, assume that $R$ is a domain. To finish the proof we only need to show that (2) implies (4). If (2) holds, then by Corollary \ref{w-torsion-free and GV-torsion-free} every $\GV$-torsion-free $R$-module is torsion-free. In particular, for any $\mathfrak{m}\in\wmax(R)$, we have that $R/\mathfrak{m}$ is a torsion-free $R$-module. However, if $\mathfrak{m}\neq 0$, then $R/\mathfrak{m}$ is also a torsion $R$-module, and so $R/\mathfrak{m}=0$, i.e., $\mathfrak{m}=R$, which is impossible. Thus, we obtain $\wmax(R)=\{0\}$. Hence, it follows from \cite[Proposition 2.6]{KLZ19} that $R$ is a DW-domain. From this, we see that $R$ is a domain over which every module is torsion-free, and consequently it is a field.
\end{proof}

Certainly, $\mathcal{WC}_{\GV}\subseteq\mathcal{EC}_{\GV}$. Moreover, we have:

\begin{theorem}\label{strongly GV-W-cotorsion=strongly GV-E-cotorsion}
	The following statements are equivalent for a ring $R$.
	\begin{enumerate}
		\item $\mathcal{WC}_{\GV}=\mathcal{EC}_{\GV}$.
		\item $\mathcal{TF}_w=\mathcal{FL}_w$
		\item $\mathcal{WC}_w=\mathcal{EC}_w$.
	\end{enumerate}
		Furthermore, if $R$ is a domain, then each of the above statements is equivalent to 
	\begin{enumerate}
		\setcounter{enumi}{3}
		\item $R$ is a P$v$MD.
	\end{enumerate}
\end{theorem}

\begin{proof}
	By Proposition \ref{w-W-cotorsion=w-E-cotorsion} it is enough to prove the equivalence of (1) and (2).
	
	$(1)\Rightarrow (2)$ Let $\mathcal{WC}_{\GV}=\mathcal{EC}_{\GV}$. Then it follows from Propositions \ref{GV-W-w-cotorsion pair} and \ref{GV-E-w-cotorsion pair} that
	\[\mathcal{TF}_w=\rule{0pt}{10pt}^{\bot_w}\mathcal{WC}_{\GV}=\rule{0pt}{10pt}^{\bot_w}\mathcal{EC}_{\GV}=\mathcal{FL}_w.\]
	
	$(2)\Rightarrow (1)$ This is obvious.
\end{proof}

\section{Strongly flat modules relative to $\tau_w$}

First, recall the following definition that is due to Trlifaj \cite[Section 2]{Trlifaj01}: an $R$-module $M$ is called \textit{strongly flat} if $\Ext^1_R(M,C)=0$ for all Matlis-cotorsion $R$-modules $C$. Let us denote by $\mathcal{SF}$ the class of all strongly flat $R$-modules. Then $\mathcal{SF}=\rule{0pt}{10pt}^{\bot}\mathcal{MC}$. It is known that $\mathcal{SF}\subseteq\mathcal{FL}$, i.e., every strongly flat module is flat, and that the pair $(\mathcal{SF},\mathcal{MC})$ the is the cotorsion pair generated by $Q$. In the lecture notes \cite[Section 5]{Trlifaj00}, Trlifaj raised the following question: when is the class of strongly flat modules over a domain a cover class?

This question was solved in the series of papers by Bazzoni and Salce \cite{BS02,BS04}, where it was proved that, for a domain $R$, $\mathcal{SF}$ is a cover class (i.e., all $R$-modules have strongly flat covers) if and only if $\mathcal{FL}=\mathcal{SF}$, and if and only if $R$ is an \textit{almost perfect domain}. The latter means that all proper factor rings of $R$ are perfect rings. 

Strongly flat modules over Matlis domains were studied in \cite{FST04,Lee15}. Moreover, Fuchs and Lee \cite{FL14} considered domains over which all the reduced strongly flat
modules are projective, and gave several characterizations of such domains.

This section is devoted to investigate strongly flat modules relative to $\tau_w$. We will see that almost perfect domains can also be characterized in terms of these modules.

\subsection{Definition and some properties of $w$-strongly flat modules}\quad

The first part of Proposition \ref{GV-E-cotorsion is M-cotorsion} motivates the following:

\begin{definition}\label{the definition of w-strongly flat}
	Let $M$ be an $R$-module. Then $M$ is said to be \textit{$w$-strongly flat} if $\Ext_R^1(M,C)$ is $\GV$-torsion for all Matlis-cotorsion $R$-modules $C$, that is, if $M\in\rule{0pt}{10pt}^{\bot_w}\mathcal{MC}$.
\end{definition}

Let us denote by $\mathcal{SF}_w$ the class of all $w$-strongly flat $R$-modules. Then $\mathcal{SF}_w=\rule{0pt}{10pt}^{\bot_w}\mathcal{MC}$. Of course, $\mathcal{SF}\subseteq\mathcal{SF}_w$, i.e., all strongly flat modules are $w$-strongly flat; and $\mathcal{SP}_w\subseteq\mathcal{SF}_w$, i.e., every $w$-split module is $w$-strongly flat. Moreover, for every $\GV$-ideal $J$ of $R$, $R/J$ is a $w$-strongly flat $R$-module, since $\Ext_R^1(R/J,-)$ is annihilate by $J$.

\begin{proposition}\label{strongly w-flat=strongly flat}
	Let $R$ be a ring. Then $\mathcal{SF}=\mathcal{SF}_w$ if and only if $R$ is a DW-ring.
\end{proposition}
\begin{proof}
	The sufficiency is obvious. For the converse, assume $\mathcal{SF}=\mathcal{SF}_w$. For each $J\in\GV(R)$, we have $R/J\in\mathcal{SF}_w=\mathcal{SF}\subseteq\mathcal{FL}$, and so $R/J$ is $\GV$-torsion-free over $R$. However, $R/J$ is also a $\GV$-torsion $R$-module, so $R/J=0$, i.e., $J=R$. Thus, it follows that $R$ is a DW-ring.
\end{proof}

\begin{proposition}\label{strongly w-flat and w-split exact sequence}
	Let $0\rightarrow A\xrightarrow[]{f} B\xrightarrow[]{g} C\rightarrow 0$ be a $w$-split exact sequence of $R$-modules with $B\in\mathcal{SF}_w$. Then $C\in\mathcal{SF}_w$.
\end{proposition}
\begin{proof}
	For any $M\in\mathcal{MC}$, we obtain the following exact sequence
	\[\Hom_R(B,M)\xrightarrow[]{~f^*~} \Hom_R(A,M)\rightarrow\Ext^1_R(C,M)\rightarrow\Ext^1_R(B,M).\]
	Since $B\in\mathcal{SF}_w$, $\Ext^1_R(B,M)$ is $\GV$-torsion. Hence, to complete the proof, we need only show that $\cok(f^*)$ is also $\GV$-torsion. But this is the same as in the proof of $(1)\Rightarrow (2)$ of \cite[Proposition 2.4]{WQ20}.
\end{proof}

\subsection{$w$-strongly flat modules over domains}\quad

Next, we investigate $w$-strongly flat modules over domains.

\begin{proposition}\label{strongly w-flat is w-flat}
	Let $R$ be a domain. Then $\mathcal{SF}_w\subseteq\mathcal{FL}_w$, i.e., all $w$-strongly flat $R$-modules are $w$-flat.
\end{proposition}
\begin{proof}
	By Proposition \ref{GV-E-cotorsion is M-cotorsion}, $\mathcal{EC}_{\GV}\subseteq\mathcal{MC}$. Thus, Proposition \ref{GV-E-w-cotorsion pair} implies
	\[\mathcal{SF}_w=\rule{0pt}{10pt}^{\bot_w}\mathcal{MC}\subseteq\rule{0pt}{10pt}^{\bot_w}\mathcal{EC}_{\GV}=\mathcal{FL}_w.\]
\end{proof}

\begin{proposition}\label{GV-M-w-cotorsion pair}
	Let $R$ be a domain. Then the pair $\left(\mathcal{SF}_w,\mathcal{MC}\right)$ is the $w$-cotorsion pair generated by $Q$.
\end{proposition}
\begin{proof}
	This follows immediately from Propositions \ref{GV-E-cotorsion is M-cotorsion} and \ref{w-cotorsion pair}(2).
\end{proof}

Let $R$ be a domain and $M$ an $R$-module. Recall from \cite[p. 3]{Matlis64} that $M$ is called \textit{$h$-divisible} if it is an homomorphic image of an injective $R$-module, and that $M$ is called \textit{$h$-reduced} if no nonzero submodule of $M$ is $h$-divisible. It is clear that $M$ is $h$-reduced if and only if $\Hom_R(Q,M)=0$. 

The following theorem gives some characterizations of $w$-strongly flat modules over domains.

\begin{theorem}\label{characterization of  strongly w-flat}
	Let $R$ be a domain and $M$ an $R$-module. Then the following statements are equivalent.
	\begin{enumerate}
		\item $M$ is $w$-strongly flat.
		\item There is a $w$-split exact sequence of $R$-modules $0\rightarrow C\rightarrow N\rightarrow M\rightarrow 0$, where $N$ fits in an exact sequence of $R$-modules $0\rightarrow F\rightarrow N\rightarrow E\rightarrow 0$ with $F$ free and $E$ a $Q$-vector space.
		\item $M\in\mathcal{FL}_w$ and $\Ext^1_R(M,C)$ is $\GV$-torsion for all $C\in\mathcal{FL}_w\bigcap\mathcal{MC}$.
	\end{enumerate}
\end{theorem}
\begin{proof}
	$(1)\Rightarrow (2)$ Let $M$ be a $w$-strongly flat module and consider a free presentation of $M$: $0\rightarrow H\rightarrow F\rightarrow M\rightarrow 0$, where $F$ is a free $R$-module. Since $F$ is a torsion-free and $h$-reduced module, $H$ is also torsion-free and $h$-reduced. Hence, by \cite[Corollary 2.2 and p. 3, (III)]{Matlis64}, there exists an exact sequence of $R$-modules $0\rightarrow H\rightarrow C\rightarrow E\rightarrow 0$, where $E$ is a $Q$-vector space and $C$ is an $h$-reduced Matlis-cotorsion module. Consider the following pushout diagram:
	\[\xymatrix{
		& 0\ar[d] & 0\ar[d] & & \\
		0\ar[r]	& H\ar[r]\ar[d] & F\ar[r]\ar[d] & M \ar[r]\ar@{=}[d]& 0 \\
		0\ar[r]	& C\ar[r]\ar[d] & N\ar[r]\ar[d] & M\ar[r] & 0 \\
		& E\ar@{=}[r]\ar[d] & E\ar[d] & & \\
		& 0 & 0 & &
	}\]
Since $M$ is $w$-strongly flat, $\Ext^1_R(M,C)$ is $\GV$-torsion; so the second row of the above diagram $w$-splits by Lemma \ref{w-split exact sequence}. Thus statement (2) holds.
	
	$(2)\Rightarrow (1)$ Assume that (2) holds. Since $F$ and $E$ are $w$-strongly flat modules, $N$ is also $w$-strongly flat. Hence, by Proposition \ref{strongly w-flat and w-split exact sequence}, $M$ is $w$-strongly flat.
	
	$(1)\Rightarrow (3)$ This is a consequence of Proposition \ref{strongly w-flat is w-flat}
	
	$(3)\Rightarrow (1)$ Suppose that (3) holds and let $C$ be any Matlis-cotorsion $R$-module. Then by Proposition \ref{w-E-cotorsion pair}, there is an exact sequence of $R$-modules
	\[0\rightarrow X\rightarrow G\rightarrow C\rightarrow 0,\]
	where $G$ is $w$-flat and $X$ is $w$-Enochs-cotorsion. Thus $G\in \mathcal{FL}_w\bigcap\mathcal{MC}$.  By the hypothesis, $\Ext^1_R(M,G)$ is $\GV$-torsion. Consider the following exact sequence of $R$-modules $0\rightarrow E\rightarrow P\rightarrow M\rightarrow 0$ with $P$ projective and $E$ $w$-flat. Hence, we obtain $\Ext^2_R(M,X)\cong\Ext^1_R(E,X)=0$. Therefore, the exact sequence
	\[\Ext^1_R(M,G)\rightarrow\Ext^1_R(M,C)\rightarrow\Ext^2_R(M,X)=0,\]
	implies that $\Ext^1_R(M,C)$ is a $\GV$-torsion module. Consequently,  $M$ is a $w$-strongly flat module.
\end{proof}

Now, we characterize the $w$-strongly flat modules within the $\GV$-torsion modules.

\begin{proposition}\label{GV-torsion strongly w-flat}
	Let $R$ be a domain and $M$ a $\GV$-torsion $R$-module. Then the following statements are equivalent.
	\begin{enumerate}
		\item $M$ is $w$-strongly flat.
		\item $JM=0$ for some $J\in\GV(R)$.
		\item $M$ is $w$-split.
	\end{enumerate}
\end{proposition}
\begin{proof}
	$(1)\Rightarrow (2)$  Suppose that $M$ is a $w$-strongly flat module. The torsion-free cover $\pi$ of $M$ yields an exact sequence
	\[0\rightarrow W\rightarrow F\xrightarrow[]{~\pi~}M\rightarrow 0,\]
	where $F$ is torsion-free (and hence $\GV$-torsion-free) and $W$ is Warfield-cotorsion (and hence Matlis-cotorsion). So $\Ext_R^1(M,W)$ is $\GV$-torsion, thus by Lemma \ref{w-split exact sequence}, the above exact sequence is $w$-split. Therefore, there exist $J=\langle d_1,d_2,\cdots,d_n\rangle\in\GV(R)$ and $h_1,h_2,\cdots,h_n\in\Hom_R(M,F)$ such that $\eta^M_{d_k}=\pi h_k$ for $k=1,2,\cdots,n$. Since $M$ is $\GV$-torsion and $F$ is $\GV$-torsion-free, $\Hom_R(M,F)=0$, which forces $h_k=0$ for $k=1,2,\cdots,n$. Consequently, for any $\forall x\in M$, we have
	\[d_kx=\eta^M_{d_k}(x)=\pi h_k(x)=0,~k=1,2,\cdots,n.\]
	So $JM=0$.
	
	$(2)\Leftrightarrow (3)$ See \cite[Lemma 2.3]{WQ20}.
	
	$(3)\Rightarrow (1)$ This is trivial.
\end{proof}

It is known that all $\GV$-torsion modules are $w$-flat and that every $w$-strongly flat module over a domain is $w$-flat. However, not all $\GV$-torsion modules are $w$-strongly flat.

\begin{example}
	Let $R$ be a two dimensional regular local ring and set
	\[M=\bigoplus\{R/J~|~J\in\GV(R)\}.\]
	Then $M$ is a $\GV$-torsion $R$-module which is not $w$-split (see \cite[Example 2.6]{WQ20}). By Proposition \ref{GV-torsion strongly w-flat}, it is not $w$-strongly flat over $R$, either.
\end{example}

We next consider the localizations of $w$-strongly flat modules over a domain $R$ at the prime $w$-ideals of $R$. 

\begin{proposition}\label{localization of strongly $w$-flat}
	Let $R$ be a domain and let $M$ be a $w$-strongly flat $R$-module. Then $M_{\p}$ is a strongly flat $R_{\p}$-module for all prime $w$-ideals $\p$.
\end{proposition}
\begin{proof}
	Suppose that $\p$ is a prime $w$-ideal of $R$. Since $M$ is a $w$-strongly flat $R$-module, Theorem \ref{characterization of  strongly w-flat} yields a $w$-split exact sequence of $R$-modules $0\rightarrow C\rightarrow N\rightarrow M\rightarrow 0$, where $N$ fits in an exact sequence of $R$-modules $0\rightarrow F\rightarrow N\rightarrow E\rightarrow 0$ with $F$ free and $E$ a $Q$-vector space. Therefore, by Lemma \ref{localization of w-split exact sequence}, the exact sequences of $R_{\p}$-modules $0\rightarrow C_{\p}\rightarrow N_{\p}\rightarrow M_{\p}\rightarrow 0$ is split, where $N_{\p}$ fits in an exact sequence of $R_{\p}$-modules $0\rightarrow F_{\p}\rightarrow N_{\p}\rightarrow E\rightarrow 0$ with $F_{\p}$ a free $R_{\p}$-module. Thus, it follows from \cite[Corollary 4.4.11]{GT06} that $M_{\p}$ is a strongly flat $R_{\p}$-module.
\end{proof}

In \cite[Proposition 2.5]{BS04}, Bazzoni and Salce showed that strongly flat submodules of a projective $R$-modules are themselves projective under the hypotheses that $R$ is a Matlis domain and the submodule is an essential submodule. Later, in \cite[Theorem 3.2]{Lee15}, Lee dropped both conditions and got the same result. 

Now, we concentrate on $w$-strongly flat submodules of projective modules. 

\begin{lemma}\label{iwsd-leq-1}
	Let $R$ be a domain and $M$ an $R$-module. Then $\wsd_RM\leqslant 1$ if and only if $\Ext_R^1(M,D)$ is $\GV$-torsion for all $h$-divisible $R$-modules $D$.
\end{lemma}
\begin{proof}
	Assume that $\wsd_RM\leqslant 1$ and let $D$ be any $h$-divisible $R$-module. Then there is an exact sequence of $R$-modules $0\rightarrow K\rightarrow E\rightarrow D \rightarrow 0$,
	where $E$ is injective. Since $\wsd_RM\leqslant 1$, it follows that
	$\Ext_R^1(M,D)\cong\Ext_R^2(M,K)$ is a $\GV$-torsion module.
	
	Conversely, suppose that $\Ext_R^1(M,D)$ is $\GV$-torsion for all $h$-divisible $R$-modules $D$. Then for each $R$-module $K$, consider the exact sequence of $R$-modules $0\rightarrow K\rightarrow E\rightarrow D \rightarrow 0$ with $E$ injective and $D$ $h$-divisible. By the hypothesis, $\Ext_R^1(M,D)$ is a $\GV$-torsion module, and hence $\Ext_R^2(M,K)\cong\Ext_R^1(M,D)$ is also $\GV$-torsion. Therefore, $\wsd_RM\leqslant 1$.
\end{proof}

\begin{lemma}\label{w-sm}
	Let $R$ be a domain, $P$ be a projective $R$-module, and $F$ be a $w$-strongly flat $R$-submodule of $P$. Then $F$ is $w$-split.
\end{lemma}
\begin{proof}
	Let $D$ be an $h$-divisible $R$-module. Then by \cite[Theorem 2.1 and p. 3, (II)]{Matlis64}, there exists an exact sequence of $R$-modules $0\rightarrow M\rightarrow E\rightarrow D\rightarrow 0$, where $E$ is a $Q$-vector space and $M\in\mathcal{MC}$. Thus
	\[\Ext_R^1(P/F,D)\cong\Ext_R^2(P/F,M)\cong\Ext_R^1(F,M)\]
	is $\GV$-torsion. By Lemma \ref{iwsd-leq-1}, $\wsd_R(P/F)\leqslant 1$, and so $F$ is $w$-split.
\end{proof} 

Recall that a nonzero ideal $I$ of a domain $R$ is called \textit{$w$-invertible} if there is a fractional ideal $B$ of $R$ such that $(IB)_w=R$, or equivalently, if $(II^{-1})_w=R$, where $I^{-1}=\{r\in Q~|~rI\subseteq R\}$ (see \cite[Proposition 5.1]{WM97}). It is known that a domain $R$ is a Krull domain if and only if every nonzero ideal of $R$ is $w$-invertible (see \cite[Theorem 5.4]{WM97}).

\begin{lemma}\label{a lemma for w-invertible ideal}
	Let $R$ be a domain. If $I$ is a $w$-invertible ideal of $R$, then:
	\begin{enumerate}
		\item $I$ is of finite type, i.e., there is a finitely generated subideal $A$ of $I$ with $I_w=A_w$.
		\item For each prime $w$-ideal $\mathfrak{p}$, $I_{\mathfrak{p}}$ is a principal ideal of $R_{\mathfrak{p}}$.
	\end{enumerate}
\end{lemma}
\begin{proof} Since $I$ is $w$-invertible, $(IB)_w=R$ for some fractional ideal $B$ of $R$.
	
	(1) Because $1\in (IB)_w$, there exists a $\GV$-ideal $J$ of $R$ with $J\subseteq IB$. But $J$ is finitely generated, so there are a finitely generated subideal $A$ of $I$ and a finitely generated fractional subideal $C$ of $B$ such that $J\subseteq AC$, i.e., $(AC)_w=R$. For any $a\in I$, $aC\subseteq R$, and hence
	\[Ra=(AC)_wa\subseteq ((AC)_wa)_w=(ACa)_w\subseteq A_w.\]
	Thus it follows that $I_w=A_w$, that is, $I$ is of finite type.
	
	(2) For each prime $w$-ideal $\mathfrak{p}$, we have $R_{\mathfrak{p}}=((IB)_w)_{\mathfrak{p}}=(IB)_\mathfrak{p}=I_{\mathfrak{p}}B_{_\mathfrak{p}}$. Therefore, $I_{\mathfrak{p}}$ is an invertible ideal of $R_{\mathfrak{p}}$, and hence it is free. Thus it follows that $I_{\mathfrak{p}}$ is a principal ideal. 
\end{proof}

It was proved in \cite[Proposition 2.2]{BS04} that every strongly flat ideal of a coherent domain is projective. In \cite[Corollary 3.4]{Lee15}, Lee showed that the coherence condition in this result is superfluous. Similarly, our next result shows that every $w$-strongly flat ideal of a domain is $w$-split.

\begin{proposition}\label{strongly w-flat ideals of integral domains}
	Let $R$ be a domain and $I$ a nonzero ideal of $R$. Then the following conditions are equivalent.
	\begin{enumerate}
		\item $I$ is $w$-strongly flat.
		\item $I$ is $w$-split.
		\item $I$ is $w$-invertible.		
	\end{enumerate}
\end{proposition}
\begin{proof}
	$(1)\Rightarrow (2)$ By Lemma \ref{w-sm}. 
	
	$(2)\Rightarrow (1)$ This is trivial, since every $w$-split module is $w$-strongly flat.
	
	$(2)\Rightarrow (3)$ Suppose that $I$ is $w$-split. Then by \cite[Proposition 2.4]{WQ20}, there exist elements $\{x_i\}_{i\in\Gamma}$ of $I$ and $J=(d_1,d_2,\cdots,d_n)\in\GV(R)$ such that for all $k=1,2,\cdots,n$, there are homomorphisms $\{f_{ki}\in I^*\}_{i\in\Gamma}$ satisfying that for each $x\in I$, almost all $f_{ki}(x)=0$ and $d_kx=\sum\limits_i f_{ki}(x)x_i$. Fix a nonzero $a\in I$. For all $k=1,2,\cdots,n$ and for all $i\in\Gamma$ with $f_{ki}(a)\neq 0$, write $x_{ki}=\frac{f_{ki}(a)}{a}$. Then for every nonzero $b\in I$, $\frac{f_{ki}(b)}{b}=\frac{f_{ki}(a)}{a}=x_{ki}$, and so $x_{ki}\in I^{-1}$. Thus, it follows that for all $k=1,2,\cdots,n$, $d_k=\sum\limits_i x_{ki}x_i\in II^{-1}$, that is, $J\subseteq II^{-1}$. Therefore, $(II^{-1})_w=R$, and (3) holds.
	
	$(3)\Rightarrow (2)$ Assume that $I$ is $w$-invertible. We first show that for every torsion-free $w$-module $N$ over $R$, $\Ext_R^1(I,N)$ is $\GV$-torsion. By Lemma \ref{a lemma for w-invertible ideal}, $I$ is of finite type and $I_{\mathfrak{m}}$ is a principal ideal of $R_{\mathfrak{m}}$ for each maximal $w$-ideal $\mathfrak{m}$ of $R$. Hence it follows from \cite[Proposition 1.11]{WK15} that the induced homomorphism $\theta_1:\Ext_R^1(I,N)_{\mathfrak{m}}\rightarrow\Ext_{R_{\mathfrak{m}}}^1(I_{\mathfrak{m}},N_{\mathfrak{m}})$ is monic. Thus, $\Ext_R^1(I,N)_{\mathfrak{m}}=0$ for each maximal $w$-ideal $\mathfrak{m}$ of $R$, i.e., $\Ext_R^1(I,N)$ is $\GV$-torsion. 
	
	Now, let $g:P\rightarrow I$ be an $R$-epimorphism with $P$ a projective module and write $L=\ker g$. Then since $I$ is $\GV$-torsion-free and $P$ is a torsion-free $w$-module, $L$ is also a torsion-free $w$-module, and hence $\Ext_R^1(I,L)$ is $\GV$-torsion. Applying $\Hom_R(I,-)$ to the exact sequence $0\rightarrow L\rightarrow P\stackrel{g}{\rightarrow} I\rightarrow 0$ gives an exact sequence
	$$\Hom_R(I,P)\stackrel{g_*}{\longrightarrow} \Hom_R(I,I)\rightarrow
	\Ext^1_R(I,L).$$ 
	Thus, it follows that $\cok g_*$ is also $\GV$-torsion,
	whence $I$ is $w$-split by the equivalence of (1) and (5) in \cite[Proposition 2.4]{WQ20}.
\end{proof} 

In \cite[Corollary 2.4]{BS04}, Bazzoni and Salce proved that a domain $R$ is a Dedekind domain if and only if every ideal of $R$ is strongly flat. As a consequence of Proposition \ref{strongly w-flat ideals of integral domains}, we have the following $w$-theoretic analogue of this result.

\begin{corollary}\label{a chracterization of Krull domians in terms of strongly $w$-flat modules}
	Let $R$ be a domain. Then the following conditions are equivalent.
	\begin{enumerate}
		\item $R$ is a Krull domain.
		\item Every ideal of $R$ is $w$-strongly flat.
		\item Every ideal of $R$ is $w$-split.
	\end{enumerate}
\end{corollary}

The following result characterizes domains over which every $w$-strongly flat module is $w$-split.

\begin{proposition}\label{strongly w-flat=w-split}
	The following statements are equivalent for a domain $R$.
	\begin{enumerate}
		\item $\mathcal{SF}_w=\mathcal{SP}_w$.
		\item $\mathcal{MC}=\m$, i.e., every $R$-module is a Matlis-cotorsion module.
		\item $\mathcal{SF}=\mathcal{P}_0$.
	\end{enumerate}
\end{proposition}
\begin{proof}
	$(1)\Rightarrow (2)$ If $\mathcal{SF}_w=\mathcal{SP}_w$, then by Proposition \ref{GV-M-w-cotorsion pair}, we obtain
	\[\mathcal{MC}={\mathcal{SF}_w}^{\bot_w}={\mathcal{SP}_w}^{\bot_w}=\m.\]
	
	$(2)\Rightarrow (1)$ and $(2)\Rightarrow (3)$ are obvious.
	
	$(3)\Rightarrow (2)$ If $\mathcal{SF}=\mathcal{P}_0$, then
	\[\mathcal{MC}=\mathcal{SF}^{\bot}={\mathcal{P}_0}^{\bot}=\m.\]
\end{proof}

\subsection{Applications of $w$-strongly flat modules}\quad

Recall that a domain $R$ is called a \textit{Matlis domain} if the projective dimension of $Q$ as an $R$-module is at most one, and that this property is equivalent to the fact that every divisible module is $h$-divisible.

\begin{theorem}\label{characterization of Matlis domain}
	The following statements are equivalent for a domain $R$.
	\begin{enumerate}
		\item $R$ is a Matlis domain.
		\item $\wsd_R Q\leqslant 1$.
		\item The class $\mathcal{SF}_w$ is resolving.
		\item The $w$-split dimension of any $w$-strongly flat $R$-module is at most one.
		\item The $w$-split dimension of any strongly flat $R$-module is at most one.
	\end{enumerate}
\end{theorem}
\begin{proof}
	$(1)\Rightarrow (2)$ This is clear.
	
	$(2)\Rightarrow (1)$ Suppose that $\wsd_RQ\leqslant 1$. Then for each $R$-module $N$, we have that $\Ext^2_R(Q,N)$ is a $\GV$-torsion $R$-module. But note that the $Q$-module $\Ext^2_R(Q,N)$, as an $R$-module, is $\GV$-torsion-free. Thus $\Ext^2_R(Q,N)=0$, and consequently $\pd_RQ\leqslant 1$, i.e., $R$ is a Matlis domain.
	
	$(1)\Rightarrow (3)$ Assume that $R$ is a Matlis domain. Then by \cite[Lemma 4.4.13]{GT06}, $\mathcal{MC}$ is coresolving. Hence, it follows from Propositions \ref{GV-M-w-cotorsion pair} and \ref{hereditary w-cotorsion pair} that $\mathcal{SF}_w$ is resolving.
	
	$(3)\Rightarrow (4)$ Suppose that the class $\mathcal{SF}_w$ is resolving, and that $M$ is a $w$-strongly flat $R$-module. Consider the exact sequence of $R$-modules $0\rightarrow F\rightarrow P\rightarrow M\rightarrow 0$, where $P$ is projective. Then by the hypothesis, $F$ is $w$-strongly flat, and so, by Lemma \ref{w-sm}, $F$ is $w$-split. Thus, $\wsd_RM\leqslant 1$.
	
	$(4)\Rightarrow (5)$ and $(5)\Rightarrow (2)$ are obvious.
\end{proof}

It is known from \cite[Theorem 4.5]{BS02} that a domain $R$ satisfies $\mathcal{SF}=\mathcal{FL}$ if and only if $\mathcal{MC}=\mathcal{EC}$, if and only if $R$ is almost perfect. Similarly, we have the following characterization of the domains satisfying $\mathcal{SF}_w=\mathcal{FL}_{w}$.

\begin{theorem}\label{M-cotorsion=GV-E-cotorsion}
	The following statements are equivalent for a domain $R$.
	\begin{enumerate}
		\item $\mathcal{MC}=\mathcal{EC}_{\GV}$.
		\item $\mathcal{SF}_w=\mathcal{FL}_{w}$.
		\item $R$ is almost perfect.
	\end{enumerate}
\end{theorem}

\begin{proof}
	$(1)\Rightarrow (2)$ If $\mathcal{MC}=\mathcal{EC}_{\GV}$, then Proposition \ref{GV-E-w-cotorsion pair} gives
	\[\mathcal{SF}_w=\rule{0pt}{10pt}^{\bot_w}\mathcal{MC}=\rule{0pt}{10pt}^{\bot_w}\mathcal{EC}_{\GV}=\mathcal{FL}_{w}.\]
	
	$(2)\Rightarrow (3)$ In view of Propositions \ref{w-E-cotorsion=E-cotorsion}, \ref{strongly w-flat=strongly flat} and \cite[Theorem 4.5]{BS02} it suffices to show that $R$ is a DW-domain. To do so, set
	\[M=\bigoplus\{R/J~|~J\in\GV(R)\}.\]
	Then $M$ is a $\GV$-torsion $R$-module, and hence it is $w$-flat over $R$. By (2), $M$ is $w$-strongly flat. Therefore, it follows from Proposition \ref{GV-torsion strongly w-flat} that $J_0M=0$ for some $J_0\in\GV(R)$. So $J_0\subseteq J$ for all $J\in\GV(R)$. Now, the rest of the proof is the same as in $(3)\Rightarrow (4)$ of the proof of Theorem \ref{perfect ring}.
	
	$(3)\Rightarrow (1)$ Let $R$ be an almost perfect domain. Then $\dim(R)=1$ by \cite[Theorem 1.4]{BS03}. Thus, $R$ is a DW-domain, and so $\mathcal{EC}_{\GV}=\mathcal{EC}$. Hence, (1) follows from \cite[Theorem 4.5]{BS02}.
\end{proof}

The final result characterizes the domains over which the classes of $\mathcal{MC}$ and $\mathcal{WC}_{\GV}$ coincide.

\begin{theorem}\label{dedekind domain}
	The following statements are equivalent for a domain $R$.
	\begin{enumerate}
		\item $\mathcal{MC}=\mathcal{WC}_{\GV}$.
		\item $\mathcal{TF}_{w}=\mathcal{SF}_w$.
		\item $R$ is a Dedekind domain.
	\end{enumerate}
\end{theorem}
\begin{proof}
	$(1)\Rightarrow (2)$ Let $\mathcal{MC}=\mathcal{WC}_{\GV}$. Then by Proposition \ref{GV-W-w-cotorsion pair}, we have
	\[\mathcal{TF}_{w}=\rule{0pt}{10pt}^{\bot_w}\mathcal{WC}_{\GV}=\rule{0pt}{10pt}^{\bot_w}\mathcal{MC}=\mathcal{SF}_w.\]
	
	$(2)\Rightarrow (1)$ If $\mathcal{TF}_{w}=\mathcal{SF}_w$, then Proposition \ref{GV-M-w-cotorsion pair} yields
	\[\mathcal{MC}={\mathcal{SF}_w}^{\bot_{w}}={\mathcal{TF}_w}^{\bot_{w}}=\mathcal{WC}_{\GV}.\]
	
	$(1)\Rightarrow (3)$ Assume that $\mathcal{MC}=\mathcal{WC}_{\GV}$. Then $\mathcal{MC}=\mathcal{EC}_{\GV}$. Thus, by Theorem \ref{M-cotorsion=GV-E-cotorsion}, $R$ is almost perfect. In particular, $R$ is a DW-domain, which means that $\mathcal{MC}=\mathcal{WC}$. Hence, it follows form \cite[Theorem 4.4.9]{GT06} that $R$ is a Dedekind domain.
	
	$(3)\Rightarrow (1)$ Suppose that $R$ is a Dedekind domain. Then $R$ is a DW-domain, and hence $\mathcal{WC}_{\GV}=\mathcal{WC}$. Thus, (1) holds by \cite[Theorem 4.4.9]{GT06}.
\end{proof}

\end{document}